\documentclass{amsart}

\usepackage[utf8]{inputenc}
\usepackage[T1]{fontenc}
\usepackage[english]{babel}
\usepackage{amsthm}
\usepackage{amsmath}
\usepackage{mathtools}
\usepackage{tikz-cd}
\usepackage{amsfonts}
\usepackage{amssymb}
\usepackage{graphicx}
\usepackage{xspace}
\usepackage{enumitem}
\usepackage{float}
\floatstyle{plaintop}
\restylefloat{table}

\usepackage[mathscr]{eucal} 
\usepackage{amsmath} 
\usepackage{epsfig}
\usepackage{amscd}
\usepackage{verbatim}
\usepackage{booktabs}

\def\bCD{\begin{tikzcd}}
\def\eCD{\end{tikzcd}}

\newtheorem{TEO}{Theorem}[section]

\newtheorem{DEF}[TEO]{Definition}

\newtheorem{REM}[TEO]{Remark}

\theoremstyle{definition}

\newtheoremstyle{dico}
 {\baselineskip}   
  {\topsep}   
  {}  
  {0pt}       
  {} 
  {.}         
  {5pt plus 1pt minus 1pt} 
  {}          
\theoremstyle{dico}
\newtheorem{say}[TEO]{}
\numberwithin{equation}{section}

\def\OO{{\mathcal O}}

\newcommand\Ga{\Gamma}

\newcommand\Z{\mathbb Z}
\newcommand\ZZ{\mathbb Z}

\newcommand\R{\mathbb R} \newcommand\Q{\mathbb Q}
\newcommand\Co{\mathbb C}

\newcommand\A{{\mathsf A}}
\newcommand\T{{\mathsf T}}

\newcommand\Pic{\operatorname{Pic}}

\newcommand{\RR}{\mathsf{R}}

\newcommand{\Nm}{\operatorname{Nm}}
\newcommand{\Tei}{{\mathsf{T}}}

\newcommand{\M}{{\mathsf{M}}}

\newcommand{\meno}{^{-1}}
\newcommand{\Aut}{\operatorname{Aut}}
\newcommand{\End}{\operatorname{End}}

\newcommand{\restr}[1]{\vert_{#1}}

\newcommand{\PP}{{\mathbb P}}

\newcommand{\eps}{\varepsilon}
\renewcommand{\phi}{\varphi}
\newcommand{\lds}{\ldots}
\newcommand{\cds}{\cdots}
\newcommand{\cd}{\cdot}

\newcommand{\sx}{\langle}
\newcommand{\xs}{\rangle}
\newcommand{\lra}{\longrightarrow}
\newcommand{\ra}{\rightarrow}
\newcommand{\ga}{\gamma}
\newcommand{\alfa}{\alpha}

\newcommand{\vacuo}{\emptyset}

\newcommand{\La}{\Lambda}

\newcommand{\GL}{\operatorname{GL}}
\newcommand{\sieg}{\mathfrak{S}}

\newcommand{\Sp}                {\operatorname {Sp}}

 \newcommand{\ag}{\mathsf{A}_g}
  \newcommand{\agd}{\mathsf{A}^{\delta}_g}

 \newcommand{\zg}{\mathsf{Z}}
 \newcommand{\rgb}{\mathsf{R}_{g,b}}
 \newcommand{\pgb}{\mathsf{P}_{g,b}}

\newcommand{\y}{\Xi}
\newcommand{\dat}{{( \tilde{G}, \ttheta, \sigma)}}

 \newcommand{\braid}{\mathbf{B}_r}

\newcommand{\Fix}{\operatorname{Fix}}

\newcommand{\TG}{{\tilde{G}}}
\newcommand{\tG}{{\tilde{G}}}
\newcommand{\tg}{{\tilde{g}}}
\newcommand{\ttheta}{{\tilde{\theta}}}
\newcommand{\Mod}{\operatorname{Mod}}

\newcommand{\tj}{\widetilde{j}}  
\newcommand{\tC} {{\tilde{C} }}
\newcommand{\MAGMA}{\texttt{MAGMA}}

\begin{document}

 \title{Shimura curves in the Prym loci of  ramified double covers}

 \author[P. Frediani]{Paola Frediani} \address{ Dipartimento di
   Matematica, Universit\`a di Pavia, via Ferrata 5, I-27100 Pavia,
   Italy } \email{{\tt paola.frediani@unipv.it}}

 \author[G.P. Grosselli]{Gian Paolo Grosselli} \address{Dipartimento di
   Matematica, Universit\`a di Pavia, via Ferrata 5, I-27100, Pavia,
   Italy } \email{{\tt g.grosselli@campus.unimib.it}}

 \begin{abstract}
We study Shimura curves of PEL type in the space of polarised abelian varieties $\A_p^{\delta}$ generically contained in the ramified Prym locus.
We generalise to  ramified double covers, the construction done in \cite{cfgp}  in the unramified case and in the case of two ramification points. Namely, we construct  families of  double covers which are compatible with a fixed group action
on the base curve. We only consider the case of one-dimensional families and where the 
quotient of the base curve by the group is $\mathbb P^1$. Using computer algebra
we obtain 184 Shimura curves contained in the (ramified) Prym loci.
   \end{abstract}

 \thanks{The authors were partially supported MIUR PRIN 2017
	``Moduli spaces and Lie Theory'' ,  by MIUR, Programma Dipartimenti di Eccellenza
	(2018-2022) - Dipartimento di Matematica ``F. Casorati'',
	Universit\`a degli Studi di Pavia and by INdAM (GNSAGA)  }

 \maketitle

 \section{Introduction}

Let  $ { \rgb}$ be the moduli space parametrising isomorphism classes of triples $[(C, \eta, B)]$ where $C$ is a smooth complex projective curve of genus $g$, $B$ is a reduced effective divisor of degree $b$ on $C$ and $\eta$ is a line bundle on $C$ such that $\eta^2={\mathcal O}_{C}(B)$. To such data it is associated a double cover of $C$, $f:\tilde{C}\rightarrow C$ branched on $B$. 

 The Prym variety
 associated to $[(C,\eta, B)]$ is the connected component containing the origin of the kernel of the norm map $\Nm_{f}: J\tilde{C} \to JC$.
For $b>0$, $\ker \Nm_{f}$ is connected.  It is a  polarised abelian variety of
 dimension $g-1+\frac{b}{2}$, denoted by $P(C,\eta,B)$ or equivalently
 $P(\tilde{C}, C)$. 

Let $\Xi$ be the restriction on $P(\tC,C)$ of the principal polarisation on $J\tilde{C}$.
For $b>0$  the polarisation $\Xi$ is of type $\delta=(1,\dots, 1, \underbrace{2,\dots,2}_{g \text{ times}})$.
If $b=0$ or $2$, $\Xi$ is twice a principal polarisation and we endow $P(\tilde{C}, C)$ with this principal polarisation.


 This defines the Prym map  $$\pgb: \rgb \longrightarrow {\A}^{\delta}_{g-1+\frac{b}{2}},\quad [(C,\eta,B)] \longmapsto  [(P(C,\eta,B), \Xi)],$$
 
 where $ {\A}^{\delta}_{g-1+\frac{b}{2}}$ is the moduli space of 
 abelian varieties of dimension $g-1+\frac{b}{2} $ with a polarization of type $\delta$. We define the Prym locus as the closure in ${\A}^{\delta}_{g-1+\frac{b}{2}}$ of the image of the map $\pgb$. 
 
The Prym map $\pgb$ is generically finite, if and only if 
$
\dim \rgb \leq \dim {\A}^{\delta}_{g-1+\frac{b}{2}},
$
 (see \cite{lange-ortega}). This holds if: either $b\geq 6$ and $g\geq 1$, or $b=4$ and $g\geq 3$, $b=2$ and $g\geq 5$,  $b=0$ and $g \geq 6$. If $b =0$  the Prym map is generically injective for $g \geq 7$  (\cite{friedman-smith}, \cite{kanev-global-Torelli}). 
If $r>0$, Marcucci and Pirola \cite{pietro-vale}, and, for the missing cases, Marcucci and Naranjo \cite{mn} and Naranjo and Ortega \cite{no} have proved the generic injectivity in all the cases except for $b=4$, $g =3$, which was studied  by Nagaraj and Ramanan, and Bardelli, Ciliberto and Verra (\cite{nagarama}, \cite{bcv}) and for which the degree of the Prym map is $3$. 
Recently, a global Torelli theorem was proved for all $g$ and $b\ge 6$ (\cite{ikeda} for $g=1$ and \cite{no1} for all $g$). 

In \cite{cfgp} an analogous question to the Coleman-Oort conjecture was formulated for the classical Prym maps and for ${\mathsf{P}_{g,2}}$. Namely, the authors asked whether there exist Shimura subvarieties of $\A_g$ that are generically
   contained in the Prym loci ${\mathsf{P}}_{g+1,0}  (\RR_{g+1,0})$ and $ {\mathsf{P}}_{g,2} (\RR_{g,2})$ for $g$ sufficiently high. 
A subvariety
 $Z\subset \A_g$ is said to be generically contained in the Prym locus
  ${\mathsf{P}}_{g+1,0}  (\RR_{g+1,0})$ if $Z \subset \overline {{\mathsf{P}}_{g+1,0}  (\RR_{g+1,0})}$,
 $Z \cap  {\mathsf{P}}_{g+1,0}  (\RR_{g+1,0}) \neq \vacuo$ and $Z $ intersects the locus of
 irreducible principally polarized abelian varieties. The same applies for $ {\mathsf{P}}_{g,2} (\RR_{g,2})$. 
 In \cite{cfgp} many examples of Shimura curves generically contained in these Prym loci were found using families of Galois covers of ${\mathbb P}^1$ and computations made with \MAGMA. All examples found are contained  in $ \A_g$ for $g \leq 12$.

In this paper we address a similar question for all ramified Prym maps $\pgb$. In fact in \cite{cf3} and \cite{cf4} the second fundamental form of the Prym map $\pgb$ for $b=0$, and for $b>0$ was studied to give an upper bound for the dimension of a germ of a totally geodesic submanifold, and hence of a Shimura subvariety of ${\A}^{\delta}_{g-1+\frac{b}{2}}$ contained in the Prym locus.

The expression of the second fundamental form of the Prym maps is similar to the one of the Torelli map studied in \cite{cpt}, \cite{cfg}, \cite{gpt}, \cite{fp}, \cite{gh} and we expect that for high values of $p$, there should not exist totally geodesic subvarieties of ${\A}^{\delta}_{p}$
generically contained in the (ramified) Prym loci. 

In this paper we use a method which is similar to the one used in \cite{cfgp} to find examples of Shimura curves contained in the (ramified) Prym loci. The idea is to find families of Galois covers $ \psi_t: \tC_t \to  \tC_t/\tG \cong {\mathbb P}^1$, with 4 branch points and such that the group $\tG$ admits a central involution $\sigma \in \tG$. If we denote by $G$ the quotient $\tG/\langle \sigma \rangle$, there exists a factorisation of the map $\psi_t$ as the composition $\pi _t\circ f_t$, where $f_t: \tC_t \to C_t = \tC_t/\langle \sigma \rangle$ is a double cover and $\pi_t: C_t \to C_t/G \cong \tC_t/\tG \cong {\mathbb P}^1$. 
Then we give a sufficient condition that ensures that the Prym varieties  $P(\tC_t, C_t)$ of the double covers  $f_t: \tC _t\to C_t$ yield a Shimura curve in ${\A}^{\delta}_{\tg-g}$, where $\tg$ (resp. $g$) denotes the genus of $\tC_t$ (resp. $C_t$). 
This is done in \ref{teo1ram}. 
The condition is that the codifferential of $\pgb$ at a generic point $[(\tC_t,C_t)]$ of the family  is an isomorphism.  If the double cover $\tC_t \to C_t$ corresponds to the point $[(C_t, \eta_t, B_t)] \in  \rgb$, the  codifferential of $\pgb$ is the multiplication map $m : (S^2H^0(C_t, K_{C_t} \otimes \eta_t))^{\tG} \longrightarrow H^0(C_t, K_{C_t}^2(B_t))^{\tG}.$ 

Since the cover $\psi_t$ is branched at 4 points, we first require that $(S^2H^0(C_t, K_{C_t} \otimes \eta_t))^{\tG}$ is one dimensional (condition \eqref{condA} in section 3). Since the Prym map may have positive dimensional fibres, we also need to require that the map $m$ is non zero (condition \eqref{condB} in section 3). 

In \cite{no1} and \cite{ikeda} it is proven that $\pgb$ is an embedding for all $b \geq 6$ and for all $g>0$, hence if $b \geq 6$, and $g>0$ condition  \eqref{condA} implies condition \eqref{condB}. 

Once condition  \eqref{condA} is true, a sufficient condition ensuring
\eqref{condB} is that $(S^2H^0(C_t, K_{C_t} \otimes \eta_t))^{\tG}$ is generated by decomposable tensors (condition \eqref{condB1} in section 3). 
This happens for example if the group $\tG$ is abelian. 

If we have \eqref{condA}, another condition that implies \eqref{condB} is the following (condition \eqref{condB2}): 
we require that the Prym  variety decomposes up to isogeny as $P(\tC_t,C_t)\sim A\times JC'_t$, 
where $A$ is a fixed abelian variety and $J C'_t$ is the jacobian of a curve $C'_t=\tilde C_t/N$ defined as a quotient of $\tilde C_t$ by a normal subgroup $N\lhd\tG$, and the family of Galois  covers $C'_t\to \mathbb P^1 = C'_t/(G/N)$ satisfies condition $(*)$ of \cite{fgp} (hence it yields a Shimura curve).

Using a  \MAGMA\  (\cite{magma}), that is explained in the Appendix and is available at:

\texttt{http://www-dimat.unipv.it/grosselli/publ/}

we find many examples of such families satisfying conditions \eqref{condA} and \eqref{condB}, hence yielding  Shimura curves generically contained in the Prym loci  $\pgb(\rgb) \subset  {\A}^{\delta}_{p}$, for $0 \leq g \leq 13$, $0 \leq b \leq 16$ and $2 \leq \tg \leq 31$ ($p=\tg-g$). These computations are summarised in Table 1 in the Appendix.  The highest value of $p = \tg-g$ that we find in Table 1 is 18. We recovered all the examples found in \cite{cfgp} and also, for $g=0$, some of the examples of \cite{fgp} (the ones with an action of a  group containing a central involution). 

Notice that in the unramified case, Mohajer  \cite{moh} showed that  families of Galois covers $\tC_t \to \tC_t/\tG \cong \mathbb P^1 $, when $\tG$ is abelian, do not give rise to high dimensional Shimura subvarieties of $\A_g$ contained in the Prym locus. 

The main result is summarised in the following 

\begin{TEO}
There are 184 families of (ramified) Pryms yielding Shimura curves of  $\A_{\tg-g}^{\delta}$  for $\tg-g \leq 18$, where $0 \leq g \leq 13$, $0 \leq b \leq 16$ and $2 \leq \tg \leq 31$. 
\end{TEO}

In Table 2 we also collected all the families that we found satisfying condition \eqref{condA}. 
We performed calculations for $\tg$ up to 7. Note that after genus 49 there are no other data satisfying condition \eqref{condA}.

In section 4 we described some of the examples and we explicitly showed that they satisfy condition \eqref{condB}. 

The paper is organised as follows: in section 2 we recall some basic facts on Prym varieties and on special (or Shimura) subvarieties of $\A_g$. In section 3 we explain our construction of special subvarieties contained in the Prym loci. In section 4 we explicitly describe some of the examples that we found and we show that they satisfy conditions \eqref{condA} and \eqref{condB}. In the Appendix we explain the  \MAGMA\  script and we give Tables 1 and 2 that summarise all the examples that we found. \\

\section*{Acknowledgements}
We would like to thank Diego Conti, Alessandro Ghigi and Roberto Pignatelli for interesting discussions on the  \MAGMA\  script. We thank Bert van Geemen for a helpful correspondence on one of the examples.

\section{Preliminaries on Prym varieties and on special subvarieties of $\ag$}
\label{Shimura-section}
Denote by $ { \rgb}$ the moduli space parametrising isomorphism classes of triples $[(C, \eta, B)]$ where $C$ is a smooth complex projective curve of genus $g$, $B$ is a reduced effective divisor of degree $b$ on $C$ and $\eta$ is a line bundle on $C$ such that $\eta^2={\mathcal O}_{C}(B)$. To such data it is associated a double cover of $C$, $f:\tilde{C}\rightarrow C$ branched on $B$, with $\tilde{C}=\operatorname{Spec}({\mathcal O}_{C}\oplus \eta^{-1})$.

 The Prym variety
 associated to $[(C,\eta, B)]$ is the connected component containing the origin of the kernel of the norm map $\Nm_{f}: J\tilde{C} \to JC$.
For $b>0$, $\ker \Nm_{f}$ is connected.  It is a  polarised abelian variety of
 dimension $g-1+\frac{b}{2}$, denoted by $P(C,\eta,B)$ or equivalently
 $P(\tilde{C}, C)$. 

Let $\Xi$ be the restriction on $P(\tC,C)$ of the principal polarisation on $J\tilde{C}$.
For $b>0$ polarisation $\Xi$ is of type $\delta=(1,\dots, 1, \underbrace{2,\dots,2}_{g \text{ times}})$.
If $b=0$ or $2$ then $\Xi$ is twice a principal polarisation and we endow $P(\tilde{C}, C)$ with this principal polarisation.


 Consider the Prym map  $$\pgb: \rgb \longrightarrow {\A}^{\delta}_{g-1+\frac{b}{2}}, \ [(C,\eta,B)] \longmapsto  [(P(C,\eta,B), \Xi)],$$
 
 where $ {\A}^{\delta}_{g-1+\frac{b}{2}}$ is the moduli space of 
 abelian varieties of dimension $g-1+\frac{b}{2} $ with a polarization of type $\delta$. We define the Prym locus as the closure in ${\A}^{\delta}_{g-1+\frac{b}{2}}$ of the image of the map $\pgb$. 
 
The codifferential of $\pgb$ at a generic point $[(C, \eta, B)]$  is given by the multiplication map
\begin{equation}
\label{dp}
(d\pgb)^* : S^2H^0(C, K_C \otimes \eta) \ra H^0(C, K_C^2(B))
\end{equation}

which is known to be surjective (\cite{lange-ortega}), therefore $\pgb$ is generically finite, if and only if 
$$
\dim \rgb \leq \dim {\A}^{\delta}_{g-1+\frac{b}{2}}.
$$
This holds if: 
either $b\geq 6$ and $g\geq 1$, or $b=4$ and $g\geq 3$, $b=2$ and $g\geq 5$,  $b=0$ and $g \geq 6$.

If $b =0$  the Prym map is
 generically injective for $g \geq 7$  (\cite{friedman-smith},
 \cite{kanev-global-Torelli}). 
If $b>0$, in \cite{pietro-vale},  \cite{mn}, \cite{no}, it is  proved the generic injectivity in all the 
 cases except for $b=4$, $g =3$, which was previously studied in \cite{nagarama}, \cite{bcv} and for which 
 the degree of the Prym map is $3$. 
Recently, a global Torelli theorem was proved for all $g$ and $b\ge 6$ (\cite{ikeda} for $g=1$ and \cite{no1} for all $g$).

\begin{say}
  \label{VHS}
  Consider a rank $2g$ lattice $\Lambda \cong \Z^{2g}$ and an alternating form $Q : \Lambda \times \Lambda \ra \Z $  of
  type $\delta = (1\dots,1,2,\dots,2)$. There exists a basis of $\Lambda$ such that the corresponding matrix is
  \begin{gather*}
    \begin{pmatrix}
      0 & \Delta_g\\
      -\Delta_g & 0
    \end{pmatrix},
  \end{gather*}
  where $\Delta_g$ is the diagonal matrix whose entries are $(1,\dots,1,2,\dots,2)$. Set $U := \Lambda \otimes \R$.  
  The Siegel space is defined as follows
  \begin{gather*}
    \sieg(U,Q) := \{J \in \GL (U) : J^2 = - I, J^* Q = Q, Q(x,Jx) >0,
    \ \forall x \neq 0 \}.
  \end{gather*}
  The symplectic group $\Sp(\Lambda,Q)$ of the form $Q$ acts on the Siegel space $\sieg(U,Q)$ by conjugation and this
  action is properly discontinuous. The moduli space of abelian varieties of dimension $g$ with a polarisation of type $\delta$ is the quotient 
  $\ag^{\delta} = \Sp(\Lambda,Q) \backslash \sieg(U,Q) $.  This is a complex analytic orbifold and also a 
  smooth algebraic stack.  We will consider 
  $\ag^{\delta}$ with the orbifold structure.  To an element $J \in \sieg(U,Q)$ we associate 
  the real torus $U/\Lambda \cong  \R^{2g} / \Z^{2g}$ endowed with the complex structure $J$ and the polarization $Q$. This gives a polarised abelian
  variety that we denote by $A_J$. On $\sieg(U,Q)$ there is a natural variation of rational Hodge
  structure, whose  local system  is $\sieg(U,Q) \times \Q^{2g}$, and that 
  corresponds to the Hodge decomposition of $\Co^{2g}$ in $\pm i$
  eigenspaces for $J$.  This gives a variation of Hodge
  structure on $\A^{\delta}_g$ in the orbifold sense.
\end{say}

\begin{say}
A special or Shimura subvariety $\zg \subseteq\agd$ is by definition a Hodge locus of
  the natural variation of Hodge structure on $\agd$ described above. 
  For the definition of Hodge loci we refer to  \S 2.3 in \cite{moonen-oort}. 
  Special subvarieties contain a dense set of CM points and they are
  totally geodesic \cite[\S 3.4(b)]{moonen-oort}. Conversely an
  algebraic totally geodesic subvariety that contains a CM point is a
  special subvariety (\cite{mumford-Shimura}, 
  \cite[Thm. 4.3]{moonen-linearity-1}). 

  Let us recall the definition of  special subvarieties of
    PEL type (see \cite[\S
  3.9]{moonen-oort}). 
  Given $J\in \sieg(U,Q)$, set  $   \End_\Q (A_{J}) := \{f\in \End( \Q^{2g}): Jf=fJ\}.$

  Fix a point $J_0 \in \sieg(U,Q)$ and set $D:= \End_\Q (A_{J_0})$.  The
  \emph{PEL type} special subvariety $\zg (D)$ is defined as the image
  in $\agd$ of the connected component of the set
  $\{J \in \sieg(U,Q): D \subseteq\End_\Q(A_J)\}$ that contains $J_0$.  By
  definition $\zg(D)$ is irreducible.
\end{say}

If $G\subseteq\Sp(\Lambda, Q)$ is a finite subgroup, denote by $\sieg(U,Q)^G$ the subset of $\sieg(U,Q)$ of fixed points by the action of $G$. 
  Set
\begin{gather}
  D_G:=\{ f\in \End_\Q (\Lambda \otimes \Q) : Jf=fJ, \ \forall J \in \sieg(U,Q)^G\}.
  \label{def-DG}
\end{gather}
Let us now state a result whose proof can be found in  \cite[\S 3]{fgp}. 
\begin{TEO}
  \label{bert}
  The subset $\sieg(U,Q)^G$ is a connected complex submanifold of $\sieg(U,Q)$.
  The image of $\sieg(U,Q)^G$ in $\agd$ coincides with the PEL subvariety
  $\zg (D_G)$.  If $J \in \sieg(U,Q)^G $, then
  $ \dim \zg(D_G) = \dim (S^2 \R^{2g})^G$ where $\R^{2g}$ is endowed
  with the complex structure $J$.
\end{TEO}

\section{Special subvarieties in the Prym loci}
\label{prymetale}

Let $\Sigma_g$ be a compact connected oriented surface of genus $g$. We denote by $T_g := \T(\Sigma_g)$  the
Teichm\"uller space of $\Sigma_g$.   
Denote by $\T_{0,r}$ the Teichm\"uller space in genus $0$  with
$r\geq 4$ marked points.   Let us recall the definition of $\T_{0,r}$ that can be found in \cite[Chap.\ 15]{acg2}. 
Let us fix $p_0,\dots,p_r\in S^2$ distinct points, denote by  $P:=(p_1,\dots,p_r)$.
Points of $\T_{0,r}$ are equivalence classes of triples 
$(\PP^1, x, [h])$ where $x = (x_1, \lds, x_r) $
is an $r$-tuple of distinct points in $\PP^1$ and $[h]$ is an isotopy
class of orientation preserving homeomorphisms
$h : (\PP^1, x) \ra (S^2 , P)$.  Two such triples $(\PP^1, x , [h])$, 
$(\PP^1, x', [h'])$ are equivalent if there is a biholomorphism
$\phi: \PP^1 \ra \PP^1$ such that $\phi(x_i) = x'_i$ for any $i$ and
$ [h] = [h'\circ \phi ]$.

Using the point $p_0 \in S^2 - P$ as base point we can fix an
isomorphism 
\begin{equation}
\label{iso}
\Ga_r := \sx \ga_1, \lds, \ga_r | \ga_1\cds \ga_r =1\xs \cong \pi_1(S^2 - P, p_0).
\end{equation}

Take a point  $t=[(\PP^1,x,[h])]\in \T_{0,r}$, and fix  an epimorphism $\theta:\Gamma_r\to G$. By Riemann's existence theorem, this gives a Galois cover $C_t\to\PP^1=C_t/G$ ramified over $x$ whose monodromy is given by $\theta$,  which is endowed with an isotopy class of homeomorphisms $[h_t]$ on a surface $\Sigma_g$ covering $S^2$.

So we get a map  $\T_{0,r}\to \T_g:t\mapsto [(C_t,[h_t])]$ and the group $G$ acts on $C_t$ and embeds in the mapping class group $\Mod_g$ of $\Sigma_g$. We denote by $G_\theta$ the image of $G$ in $\Mod_g$.
The image of the map  $\T_{0,r}\to \T_g$ is the subset  $\T_g^{G_\theta}$ of  $\T_g$  consisting of the fixed points by the action of $G_\theta$ (\cite{gavino}).
By  Nielsen realization theorem, we know that  $\T_g^{G_\theta}$ is a complex submanifold  of $\T_g$ of dimension $r-3$.

The image of
$ \T^{G_\theta}$ in the moduli space $\mathsf{M}_g$ is a $(r-3)$-dimensional
algebraic subvariety (see e.g. \cite{gavino,baffo-linceo,clp2} and \cite [Thm. 2.1]{broughton-equi}).

We denote this image by $\M_g(G, \theta)$. To determine $\M_g(G, \theta)$  we have chosen the isomorphism \eqref{iso}
$\Ga_r \cong \pi_1(S^2 -P, p_0)$. To get rid of this choice, we introduce the braid group
\begin{gather*}
  \braid:=\langle \tau_1, \ldots ,\tau_{r-1}| \, \tau_i \tau_j =
  \tau_j \tau_i \, \, \text{for} \, \, |i-j|\geq 2, \,
  \tau_{i+1}\tau_i\tau_{i+1}=\tau_i\tau_{i+1}\tau_i \rangle.
\end{gather*}
There is a morphism $\phi : \braid \ra \Aut(\Ga_r)$ defined as
follows:
\begin{gather*}
  \phi( \tau_i) (\gamma_i) = \gamma_{i+1}, \quad \phi(\tau_i)
  (\gamma_{i+1}) = \gamma_{i+1} ^{-1} \gamma_i \gamma_{i+1}, \\
  \phi(\tau_i) (\gamma_j ) = \gamma_j \quad \text{for }j \neq i, i+1.
\end{gather*}
So we get an action of $\braid$ on the set of pairs $(G, \theta)$ , given by 
$ \tau \cd  (G, \theta) : = (G, \theta \circ
\phi(\tau\meno))$. 
Also the group $\Aut(G)$ acts on the set of pairs $(G, \theta)$ by
$\alfa \cd (G, \theta) : = (G, \alfa \circ \theta )$.  The
orbits of the $\braid \times \Aut(G)$--action are called \emph{Hurwitz
  equivalence classes} and elements in the same orbit are said to be
related by a \emph{Hurwitz move}. Data in the same orbit give rise to
 the same subvariety
of $\M_g$, hence the 
subvariety $\M_g(G, \theta)$ is well-defined.  For more details see
\cite{penegini2013surfaces,baffo-linceo,birman-braids}.

\begin{DEF}
\label{prymdatum}
A  Prym datum of type $(r,b)$ is a triple  $\Xi=(\tG,\tilde{\theta},\sigma)$, where $\tG$ is a finite group, $\tilde{\theta}:\Gamma_r\to\tG$ is an  epimorphism and  $\sigma\in\tG$ is a central involution and 
$$b=\sum_{i:\sigma\in\langle{\tilde\theta(\gamma_i)}\rangle}\frac{|\tG|}{\operatorname{ord}(\tilde{\theta}(\gamma_i))}.$$
\end{DEF}

Let us fix a Prym datum $\Xi=(\tG,\tilde{\theta},\sigma)$, and set  $G=\tG/\langle \sigma \rangle$.  The composition of $\tilde{\theta}$ with the projection $\tG \to G$ is an epimorphism $\theta:\Gamma_r\to G$. To  a point $t\in \T_{0,r}$ we can associate  two Galois covers  $\tC_t \to \PP^1 \cong \tC_t/\tG$ and $C_t \to \PP^1 \cong C_t/G$. Denote by $\tg$ the genus of $\tC_t$ and by $g$ the genus of $C_t$. We have a diagram 
\begin{equation}
\label{tc-c}
  \begin{tikzcd}[row sep=tiny]
    \tilde{C}_t \arrow{rr}{f_ t} \arrow{rd} & &  \arrow{ld } C_t   = \tilde{C}_t /\sx \sigma\xs  \\
    & \PP^1 &
  \end{tikzcd}
\end{equation}

and by the definition of $b$,  we immediately see that  the double covering $f_t$  has $b$ branch points. 
The covering $f_t$ is determined by its branch locus $B_t$ and by an element $\eta_t \in \Pic^{\frac{b}{2}}(C_t)$ such that $\eta_t^2 = \OO_{C_t}(B_t)$.  Hence we have a map 
$\T_{0,r}\to\rgb$
which associates to $t$ the class $[(C_t,\eta_t,B_t)]$.
This map depends on the datum $\Xi=(\tG,\tilde{\theta},\sigma)$. Denote by  $R(\Xi)$ its image. The map $\T_{r,0} \to  R(\Xi)$ has discrete fibres, hence  $\dim R(\Xi)=r-3$.

 As explained in \cite[p. 79]{gavino} there is an
  intermediate variety $\widetilde{\M}_{\tg}(\tG, \tilde{\theta})$ such that the projection ${\T}_{\tg}^{\tG_{\tilde{\theta}}} \to  \M_{\tg}(\tG, \tilde{\theta})$  factors through ${\T}_{\tg}^{\tG_{\tilde{\theta}}} \to \widetilde{\M}_{\tg}(\tG, \tilde{\theta}) \to  {\M_{\tg}}(\tG, \tilde{\theta})$. 
 
  The variety $\widetilde{\M}_{\tg}(\tG, \tilde{\theta})$ is the normalisation of $\M_{\tg}(\tG, \tilde{\theta})$ and it
  parametrises equivalence classes of curves with an action of $\tG$ of
  topological type determined by the datum $\Xi$.  We also have a map $\widetilde{\M}_{\tg}(\tG, \tilde{\theta}) \to R(\Xi)$, which associates to the class $[\tC]$ of a curve with a fixed $\tG$-action the class of the cover $\tC \to C= \tC/ \langle \sigma \rangle$, where $\sigma \in \tG$ is the central involution fixed by the datum.  Clearly the map $ {\T}_{\tg}^{\tG_{\tilde{\theta}}}  \cong \T_{0,r}  \to R(\Xi)$  is the composition  $ {\T}_{\tg}^{\tG_{\tilde{\theta}}} \to \widetilde{\M}_{\tg}(\tG, \tilde{\theta}) \to R(\Xi)$. Moreover the Prym map $\pgb$ lifts to a map $\widetilde{\M}_{\tg}(\tG, \tilde{\theta}) \to \A_{\tg-g}^{\delta}$, which sends the class of a curve $[\tC]$ with an action of $\tG$ to $P(\tC, C = \tC/\langle \sigma \rangle)$. We still denote  this map by $\pgb: \widetilde{\M}_{\tg}(\tG, \tilde{\theta}) \to \A_{\tg-g}^{\delta}$. 
      We have the following diagram

\[\bCD
{\T}_g^{G_\theta} \dar & \T_{0,r}\ar[l,"\cong" swap]\rar{\cong} \dar &{\T}_{\tg}^{\tG_{\tilde{\theta}}}\dar\\
\M_g(G , \theta) &R(\Xi)\lar& \widetilde{\M}_{\tg}(\tG, \tilde{\theta})\lar\ar[d]\\
& & {\M}_{\tg}(\tG, \tilde{\theta}) \\
\eCD\quad\bCD[column sep=tiny]
\eCD
\]

Given a Prym datum $\Xi=(\tG,\tilde{\theta},\sigma)$, let us fix an element $\tilde{C}$ of the
family ${\T}_{\tg}^{\tG_{\tilde{\theta}}}$ with the double covering
$f: \tilde{C} \lra C$. 

Set
\begin{gather*}
  V: = H^0(\tilde{C}, K_{\tilde{C}}),
\end{gather*}

We have an action of $\langle \sigma \rangle$ on $V$ and an eigenspace decomposition for this action: 
$$V: = H^0(\tilde{C}, K_{\tilde{C}})= V_+ \oplus V_-,$$
where $V_+\cong H^0(C, K_C)$ and $V_-\cong H^0(C,K_C\otimes \eta)$ as
$G$-modules.  Similarly  set $ W: = H^0( \tC , 2K_\tC )= W_+\oplus W_-$,
$W_+ \cong H^0(2K_C \otimes \eta^2) = H^0(2K_C(B))$ and
$W_- \cong H^0(C, 2K_C\otimes \eta)$.  

The multiplication map
$ m : S^2 V \ra W $ is the dual of the differential of the Torelli map
$\tj : \M_\tg \ra \A_\tg$ at the point $[\tC] \in \M_\tg$.  This map is
$\tG$-equivariant.  We have the following isomorphisms
\begin{gather*}
  (S^2V)^\tG = (S^2V_+)^{\tG} \oplus (S^2 V_-)^{\tG},  \quad
  W^\tG = W_+ ^{\tG}.
\end{gather*}

Hence the multiplication map $m$ maps $(S^2V)^\tG $ to $W_+^{\tG}$.  

Notice that the codifferential of the Prym map  at the point $[(C, \eta, B)]$  is the multiplication map in  \eqref{dp}, which coincides with the restriction of the multiplication map $m$ to $S^2 V_{-}$. Here we are  interested in the
restriction of $m$ to $ (S^2 V_-)^{\tG} $ that for simplicity we still denote by $m$: 
\begin{gather}
  \label{nostramram}
  m : (S^2 V_-)^{\tG} \lra W_+^{\tG}.
\end{gather}
By what we have said, this is the multiplication map
\begin{gather*}
  (S^2 H^0 (C, K_C\otimes \eta))^G \lra H^0(C, 2K_C \otimes \eta^2)^G
  \cong H^0(2K_C)^G \cong H^0(2K_{\tC})^{\tG}.
\end{gather*}

Now we prove a result which is a generalisation of Theorems 3.2 and 4.2 in \cite{cfgp}. 

\begin{TEO}
  \label{teo1ram}
  Let $\y=\dat$ be a Prym datum. If for some $t\in \T_{0,r}$
  the map $m$ in \eqref{nostramram} is an isomorphism, then the
  closure of $\pgb (  \widetilde{\M}_{\tg}(\tG, \tilde{\theta}))$ in ${\A}^{\delta}_{g-1+\frac{b}{2}}$ is a special
  subvariety.
\end{TEO}
\begin{proof}
  Over the Teichm\"uller space $ \Tei_{0,r}$ we have the families $\tC_t$, $C_t$, $ \eta_t$,
  $B_t$.  We also have the lattice $H_1(\tilde{C}_t, \Z)$,  the intersection form $E$
  on $H_1(\tilde{C}_t, \Z)$ and the sublattice
  $\La : = H_1(\tilde{C}_t, \Z) _- $, which  are independent of $t$.  
  
  Moreover
  $ Q:=E \restr{\La}$ is an integer-valued form on $\La$  of type $\delta = (1,\ldots,1,2,\ldots,2)$.

  Set $p := g-1+\frac{b}{2}$. 
  Let $\sieg(U,Q) $ be the Siegel space parametrizing complex
  structures on $U:= \La \otimes \R = H_1(\tilde{C}_t, \R)_-$ that are
  compatible with $Q$.  For any $t\in \T_{0,r}$ we have a
  decomposition
  $ H^1(\tilde{C}_t, \Co)_- = V_{t,-} \oplus \overline{V_{t,-}}$.
  Dualizing we get a decomposition
  $ H_1(\tilde{C}_t, \Co)_- = V^*_{t,-} \oplus \overline{V^*_{t,-}} $
  corresponding to a complex structure $J_t$ on
  $H_1(\tilde{C}_t, \R)_-$. The complex structure $J_t$ is compatible with $Q$, hence
  it is a point of $\sieg(U,Q)$, that we denote by $\phi(t)$.  So we have 
  a map $ \phi: \T_{0,r} \ra \sieg(U,Q)$ and a commutative  diagram:

  \begin{equation}
    \label{diamma2}
    \begin{tikzcd}
      \Tei_{0,r} \arrow{rr}{\phi} \arrow{d} &[-5ex]& \sieg(U,Q)
      \arrow{d}  \\
     \widetilde{\M}_{\tg}(\tG, \tilde{\theta})&\arrow{r}{\pgb}& \A^{\delta}_{p}.
    \end{tikzcd}
  \end{equation}
  The diagram commutes since 
  \begin{gather*}
    P(\tC_t, C_t) = { V_{t,-}^* }/ { \La},
  \end{gather*}
  (see e.g. \cite{lange-ortega}).

  Since the group $\TG$ preserves $E$, $\tG$
  preserves $Q$, so $\tG$ maps into $\Sp(\La, Q)$.  Denote by $G'$ the
  image of $\tG$ in $\Sp(\La, Q)$.  The complex structure $J_t$ is
  $\tG$-invariant, which is equivalent to say that  $\phi(t)=J_t \in \sieg(U,Q)^{G'}$. Hence by Theorem
  \ref{bert},  $ P (\tC_t,C_t) $ lies in the PEL special subvariety
  $\zg(D_{G'})$.  
  
  So we have shown that $\pgb (  \widetilde{\M}_{\tg}(\tG, \tilde{\theta}))\subset \zg(D_{G'})$. 
  Since $\phi (\T_{0,r}) \subset \sieg(U,Q)^{G'}$ we can view $\phi$ as a
  map $\phi: \T_{0,r} \ra \sieg(U,Q)^{G'}$.  Now, recall that
  \begin{gather*}
    \Omega^1_{\phi(t)} (\sieg(U,Q)^{G'}) \cong
    (S^2 H^0 (C_t, K_{C_t}\otimes \eta_t))^G  =  (S^2 V_{t,-})^{\tG},\\
    \Omega^1 _t \Tei_{0,r} = \Omega^1 _{[C_t]} \Tei_g^{G_\theta} =
    H^0(C_t, 2K_{C_t})^G \cong H^0(C_t, 2K_{C_t} \otimes \eta_t^2)^G =
    W_{t,+}^{\tG}.
  \end{gather*}
  The codifferential is simply the multiplication map
  \begin{gather*}
    m = (d\phi_t)^* : (S^2 V_{t,-})^{\tG} \lra W_{t,+}^{\tG}
  \end{gather*}
  (see \cite{nagarama} Prop. 3.1, or \cite{lange-ortega}).  By
  assumption there is some $t\in \T_{0,r}$ such that the map $m$ is an
  isomorphism at $t$.  So  $ \dim (S^2 V_{t,-})^{\tG} = \dim W_{t,+}^{\tG} = r-3$ and $\phi$ is an
  immersion at $t$, hence its image has dimension $r-3$.  The
  vertical arrows in \eqref{diamma2} are discrete maps, hence both
  $\pgb (  \widetilde{\M}_{\tg}(\tG, \tilde{\theta}))$ and $\zg(D_{G'}) $ have dimension $r-3$.  Since
  $\pgb(  \widetilde{\M}_{\tg}(\tG, \tilde{\theta}) )\subset \zg(D_{G'})$ and $ \zg(D_{G'})$ is irreducible, 
  we conclude that $\overline{\pgb(  \widetilde{\M}_{\tg}(\tG, \tilde{\theta}))} = \zg(D_{G'})$, so it is a special subvariety of $\A_p^{\delta}$. 
\end{proof}

We would like to use Theorem \ref{teo1ram} to construct Shimura curves contained in the closure of  $\pgb(\rgb)$ in ${\A}^{\delta}_{g-1+\frac{b}{2}}$ and intersecting $\pgb(\rgb)$.  So  from now on we will assume that $r=4$, that means that the coverings $\tC\to \tC/\tG \cong \PP^1$ and $ C \to C/G \cong \PP^1$ in diagram \eqref{tc-c} are ramified over $4$ points in $\PP^1$.

Hence, as in \cite{cfgp}, the hypothesis in Theorem \ref{teo1ram}
can be divided in two conditions: 
\begin{gather}
  \label{condA}
  \tag{A} \dim (S^2V_-)^\TG = 1.
  \\
  \label{condB}
  \tag{B} m : (S^2V_-)^\TG \lra W_+^G \text { is not identically } 0.
\end{gather}
Once condition  \eqref{condA} is true, a sufficient condition ensuring
\eqref{condB} is the following
\begin{gather}
  \label{condB1}
  \tag{B1} (S^2V_-)^\TG \text{ is generated by a decomposable tensor}.
\end{gather}

In fact, assume that $(S^2V_-)^\TG$ is generated by $
s_1 \otimes s_2 $
with $s_i \in V_-$, then $m(s_1 \otimes s_2) = s_1 \cd s_2$, which
is different from zero.

\begin{REM}
In \cite{cfgp}, Remark 3.3,  it is shown that if  \eqref{condA} holds, then \eqref{condB1} is equivalent
  to the fact that $(S^2 V_-)^{\tG} = W_1 \otimes W_2$ where the $W_i$'s are 1-dimensional representations. 
  
  Hence, if $\tilde{G}$ is abelian and condition \eqref{condA}  holds, then condition \eqref{condB1} is automatically satisfied,  since all the irreducible representations of an abelian group are 1-dimensional (see  \cite{cfgp} Remark 3.4). 
  
\end{REM}

If \eqref{condA} holds,  \eqref{condB} is equivalent to say that the Prym map is not constant.
In particular this happens if the prym decomposes up to isogeny as 
\begin{equation}\label{condB2}\tag{B2}
P(\tilde C,C)\sim A\times JC'
\end{equation}
where $A$ is a fixed abelian variety and $J C'$ is the jacobian of a curve $C'=\tilde C/N$ defined as a quotient of $\tilde C$ by a normal subgroup $N\lhd\tG$, and the family of Galois  covers $C' \to \mathbb P^1 = C'/(G/N)$ satisfies condition $(*)$ of \cite{fgp}, hence it describes a Shimura curve. 

In fact, since $A$ is a fixed abelian variety and $JC'$ moves in a Shimura curve, the Prym variety $P(\tC, C)$ necessarily moves. 

Notice that  \cite [Thm. 3.8]{fgs}) ensures  that condition  \eqref{condB2}  implies that the family  of Pryms $P(\tC,C)$ yields a totally geodesic subvariety of $\A^{\delta}_p$.

Finally in \cite{no1} it is proven that $\pgb$ is an embedding for all $b \geq 6$ and for all $g>0$, hence if $b \geq 6$, and $g>0$ condition  \eqref{condA} implies condition \eqref{condB}.

\section{Examples in the Prym locus}

In this section we describe some examples of families of covers satisfying conditions \eqref{condA} and \eqref{condB}. They do not satisfy  \eqref{condB1}, and we explicitly show that \eqref{condB} holds. However, notice that the first 5 examples  are in $\rgb$ with $b \geq 6$, hence condition \eqref{condB} is automatically satisfied thanks to \cite{no1}.

In the first two examples we show that condition \eqref{condB2}  is satisfied.

The other examples do not satisfy \eqref{condB1}, nor \eqref{condB2}. 

In particular, example 6 does not satisfy \eqref{condB1}, nor \eqref{condB2} and it has $b = 4$, hence we need to show that \eqref{condB} holds with an ad-hoc argument.

To describe the examples we first give the values of $\tg$, $g$ and $b$. We also indicate a number $\#$ that refers to the numbering of the examples that we found with the  \MAGMA\ script with the given values of $\tg$, $g$ and $b$.  For the group $\tG$ we use the presentation  and the name of the group given by  \MAGMA. 
We then give the decomposition of $V_+$ and $V_{-}$ as a direct sum of irreducible representations of $\tG$ and  also  for the  irreducible representations we use the notation of \MAGMA. \\

\noindent \textbf{Example 1.}\\
$\tg=10$, $g=4$, $b=6$, $\# 5$.\\
$\tG=G(24,10)=\ZZ/3\times D_4=\sx{g_3|g_3^3=1}\xs\times\sx{g_1,g_2,g_4|g_1^2=g_2^2=g_4^2=1, g_1^{-1}g_2g_1=g_2g_4}\xs$,\\
$\ttheta(\gamma_1)=g_1,\ \ttheta(\gamma_2)=g_2g_4,\ \ttheta(\gamma_3)=g_3^2g_4,\ \ttheta(\gamma_4)=g_1g_2g_3g_4,\quad\sigma=g_4$.\\
$V_+=V_5\oplus V_6\oplus V_7\oplus V_{10}$ and $V_-=V_{13}\oplus 2V_{15}$, where $\dim V_i=1$, $\forall i <13$, $\dim V_{13} = \dim V_{15} = 2$. 
$(S^2V_+)^\tG=V_5\otimes V_7$ and $(S^2V_-)^\tG=(S^2V_{13})^\tG$ have dimension 1. Thus condition \eqref{condA} is satisfied, but \eqref{condB1} is not.
We want to show that condition \eqref{condB2} holds and how it works. If we find a curve $C'$ such that $H^{1,0}(JC')=V_{13}$, then we obtain $P(\tC,C)\sim JC'\times A$ for some fixed abelian variety $A$ such that $H^{1,0}(A)=2V_{15}$.

Notice that $V_{13}=\Fix\sx{g_3}\xs$ and $\sx{g_3}\xs\lhd\tG$, so we consider $G'=\tG/\sx{g_3}\xs$ and $C'=\tC/\sx{g_3}\xs$.
Thus we have a commutative diagram
\[\bCD[column sep=0]
\tC\ar[rr,"\beta"]\ar[rd]&&C'=\tC/\sx{g_3}\xs\ar[dl]\\
&\PP^1=\tC/\tG=C'/G'
\eCD\]


The map $\beta$ has 6 branch points, and $C'$ has genus $2$, and the map $C'\ra C'/G'$ has 4 critical values. 
Since $\dim(S^2H^{1,0}(C'))^{G'}=\dim(S^2V_{13})^\tG=1$ then condition $(*)$ of \cite{fgp} is satisfied, thus the family of jacobians  $JC'$ yields a Shimura curve. In fact it is family $(29) = (4)$ of \cite{fgp}. 

Finally,  the group algebra decomposition for $J\tC$ (see \cite{rojas}), allows to give a decomposition up to isogeny: 
\[J\tC\sim B_5\times B_6\times B_{10}\times B_{13}^2\times B_{15}^2\]
where $\dim B_5=\dim B_{15}=2$ and the others have dimension 1. In particular $P(\tC,C)\sim B_{13}^2\times B_{15}^2$ and $J C'\sim B_{13}^2$.
In fact $JC'=E^2$ where $E$ is the elliptic curve defined by $E:=\tC/\sx g_1, g_3\xs=C'/\sx g_1\xs$, so $B_{13} \sim E$. Thus $P(\tC, C) \sim A \times JC'$, where $A := B_{15}^2$ is a fixed abelian variety and $JC' \sim E^2$ varies in the Shimura curve (29)=(4) of \cite{fgp}, hence condition  \eqref{condB2} holds.

\medskip
\noindent \textbf{Example 2.}\\
$\tg=10$, $g=4$, $b=6$, $\# 8$.\\
$\tG=G(36,12)=\ZZ/3\ZZ\times D_6=\sx{g_3|g_3^3=1}\xs\times\sx{g_1,g_2,g_4|g_1^2=g_2^2=g_4^3=1, g_1^{-1}g_4g_1=g_4^2}\xs$.\\
$\ttheta(\gamma_1)=g_1g_2,\ \ttheta(\gamma_2)=g_1g_4,\ \ttheta(\gamma_3)=g_3^2g_4^2,\ \ttheta(\gamma_4)=g_2g_3,\quad\sigma=g_2$.\\
$V_+=V_6\oplus V_9\oplus V_{16}$ and $V_-=V_5\oplus V_8\oplus V_{14}\oplus V_{15}$, where $\dim V_i=1$, for $i=5,6,8,9$, $\dim V_{14} = \dim V_{15} = \dim V_{16} = 2$. 
 .\\
$(S^2V_+)^\tG=V_6\otimes V_9$ and $(S^2V_-)^\tG=(S^2V_{14})^\tG$ have dimension 1. 
Thus condition \eqref{condA} is satisfied, but \eqref{condB1} is not.
We want to show that condition \eqref{condB2} holds. If we find a curve $C'$ such that $H^{1,0}(JC')=V_{14}$, then we obtain $P(\tC,C)\sim JC'\times A$ for some fixed abelian variety $A$ such that $H^{1,0}(A)=V_5\oplus V_8\oplus V_{15}$.

We have: $V_{14}=\Fix\sx{g_3}\xs$ and $\sx{g_3}\xs\lhd\tG$, so consider $G'=\tG/\sx{g_3}\xs$ and $C'=\tC/\sx{g_3}\xs$.
We have  a commutative diagram
\[\bCD[column sep=0]
\tC\ar[rr,"\beta"]\ar[rd]&&C'=\tC/\sx{g_3}\xs\ar[dl]\\
&\PP^1=\tC/\tG=C'/G'
\eCD\]
The map $\beta$ has 6 branch points, $C'$ has genus $2$ and the cover $C' \ra C'/G'$ has 4 critical values. 
Since $\dim(S^2H^{1,0}(C'))^{G'}=\dim(S^2V_{14})^\tG=1$, condition $(*)$ of \cite{fgp} is satisfied, thus the family of jacobians  $JC'$ yields a Shimura curve, which corresponds to family (30) of \cite{fgp}. 

The group algebra decomposition for $J\tC$ gives
\[J\tC\sim B_5\times B_6\times B_8\times B_{14}^2\times B_{15}^2\times B_{16}^2\]
where $\dim B_6=2$ and the others have dimension 1. So $P(\tC,C)\sim  A \times JC' \sim B_5\times B_8\times B_{14}^2\times B_{15}^2$, $J C'\sim B_{14}^2$, $A \sim B_5 \times B_8 \times B_{15}^2$.

\medskip
\noindent \textbf{Example 3.}\\
$\tg=10$, $g=4$, $b=6$, $\# 7$.\\
$\tG=G(36,12)=\ZZ/6\times S_3=\sx{g_2,g_3|g_2^2=g_3^3=1}\xs\times\sx{g_1,g_4|g_1^2=g_4^3=1,g_1^{-1}g_4g_1=g_4^2}\xs$,\\
$\ttheta(\gamma_1)=g_1g_4,\ \ttheta(\gamma_2)=g_1g_2g_4^2,\ \ttheta(\gamma_3)=g_3^2,\ \ttheta(\gamma_4)=g_2g_3g_4^2,\quad\sigma=g_2$.\\
$V_+=V_6\oplus V_9\oplus V_{18}$ and
$V_-=V_5\oplus V_8\oplus V_{15}\oplus V_{17}$, where $\dim V_5=\dim V_6=\dim V_8=\dim V_9=1$ and $\dim V_{15}=\dim V_{17}=\dim V_{18}=2$.\\
$(S^2V_-)^\tG=(V_{15}\otimes V_{17})^\tG$ is one-dimensional, so condition \eqref{condA} is satisfied, but neither \eqref{condB1} nor \eqref{condB2} hold.

$ G = \tG/\langle g_2 \rangle \cong \Z/3 \times S_3 = G(18,3)$, $(S^2V_+)^\tG = (V_6 \otimes V_9)$, which has dimension 1,  hence the family $C \to C/G$ satisfies condition $(*)$ of \cite{fgp} and it yields a Shimura curve, corresponding to family (38) of \cite{fgp}.  

The group algebra decomposition of $J\tC$ is the following: 
$J\tilde C\sim B_5\times B_6\times B_8\times B_{15}^2\times B_{18}^2$
where $\dim B_6=\dim B_{15}=2$ and $\dim B_5=\dim B_8=\dim B_{18}=1$.  Hence \[P(\tilde C,C)\sim B_5\times B_8\times B_{15}^2.\]
 The trace of ${g_1}_{|V_{15}}$ is zero, hence $V_{15}$ decomposes as a direct sum of the two one dimensional ($\pm 1$)-eigenspaces of the action of $g_1$. The same happens for the action of $g_1$ on $V_{17}$. Denote by $\langle e_1\rangle$ (resp.\ $\langle f_1\rangle$) the $(+1)$-eigenspace  of  ${g_1}_{|V_{15}}$ (resp.\ of  ${g_1}_{|V_{17}}$) and by $\langle e_2\rangle$ (resp.\ $\langle f_2\rangle$) the $(-1)$-eigenspace  of  ${g_1}_{|V_{15}}$ (resp.\ of  ${g_1}_{|V_{17}}$), so $V_{15}=\sx{e_1,e_2}\xs$, $V_{17}=\sx{f_1,f_2}\xs$. 
 One easily checks that the elliptic curve $B_{18}$ is  isogenous to $\tC/\langle g_1, g_2 \rangle$ and we denote by $\langle v_1 \rangle = H^{1,0}(B_{18}) \subset V_{18}$.  
We define $C':=\tilde C/\sx{g_1}\xs$ and we calculate $H^{1,0}(C')=V_8\oplus\sx{e_1}\xs\oplus\sx{f_1}\xs\oplus\sx{v_1}\xs$ and the decomposition $J C'\sim B_8\times B_{15}\times B_{18}$.
Similarly if we define $D':=\tilde C/\sx{g_1g_2}\xs$ we have $H^{1,0}(D')=V_5\oplus\sx{e_2}\xs\oplus\sx{f_2}\xs\oplus\sx{v_1}\xs$ and $J D'\sim B_5\times B_{15}\times B_{18}$.
Thus we obtain \[J C'\times J D'\sim P(\tilde C,C)\times B_{18}^2\]
and we obtain that $P(\tilde C,C)$ moves if and only if the product of jacobians on the left move since $B_{18}$ is a fixed abelian variety. Moreover it moves if and only if $B_{15}$ does, since both $B_5$ and $B_8$ are fixed.

We consider the following covers
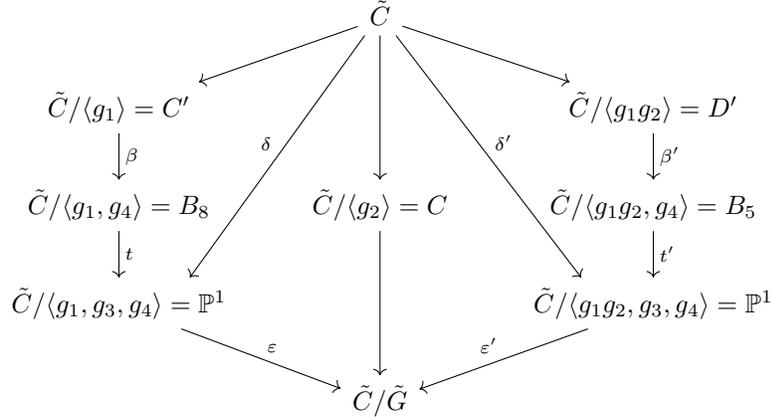
\[\bCD
&\ar[ld]\ar[lddd, end anchor={[xshift=4ex]},"\delta" swap]\tC\ar[rd]\ar[rddd, end anchor={[xshift=-4ex]},"\delta'"]\ar[dd]\\
\tC/\sx{g_1}\xs=C'\ar[d,"\beta"]&&\tC/\sx{g_1g_2}\xs=D'\ar[d,"\beta'"]\\
\tC/\sx{g_1,g_4}\xs=B_8\ar[d,"t"]&\tC/\sx g_2\xs=C\ar[dd]&\tC/\sx{g_1g_2,g_4}\xs=B_5\ar[d,"t'"]\\
\tC/\sx{g_1,g_3,g_4}\xs=\PP^1\ar[rd,"\eps"]&&\tC/\sx{g_1g_2,g_3,g_4}\xs=\PP^1\ar[ld,"\eps'" swap]\\
&\tC/\tG
\eCD\]
Assume that the critical values of the cover $\tC \to \tC/\tG \cong \PP^1$ are $\{ P_1 = \lambda, P_2 = 0, P_3 = 1, P_4 = \infty \}$. 
The maps $\eps$ and $\eps'$ are 2:1, ramified respectively over $\{0, \infty\}$, and $\{\lambda, \infty\}$. So  we can assume $\eps(z)=z^2$, and  $\eps'(z)=(1-\lambda)z^2+\lambda$. Thus the covers $\delta,\delta'$ are ramified over $1,-1,\mu,-\mu,\infty$ and $1,-1,\alpha,-\alpha,\infty$ respectly, where $\mu^2=\lambda$ and $\alpha^2=\frac{\lambda}{\lambda-1}$.
The 3:1 covers $t,t'$ are ramified over 3 points, which we can assume to be $1,-1,\infty$ for both maps. So $B_5$ and $B_8$ are isomorphic and an equation for them as $\Z/3$-covers of $\PP^1$ is given by $y^3=(x^2-1)$.
So the critical values of the maps $\beta$  (resp. $\beta'$) are the 6 preimages of $\pm\mu$, (resp. $\pm \alpha$) via $t$ (resp. $t'$). Since they move, varying $\lambda$,  the jacobians $JC' \sim JD'$ move.

\medskip
\noindent \textbf{Example 4.}\\
$\tg=10$, $g=4$, $b=6$, $\# 9$.\\
$\tG=G(36,12)=\ZZ/6\times S_3=\sx{g_2,g_3|g_2^2=g_3^3=1}\xs\times\sx{g_1,g_4|g_1^2=g_4^3=1,g_1^{-1}g_4g_1=g_4^2}\xs$,\\
$\ttheta(\gamma_1)= g_1g_4^2,\  \ttheta(\gamma_2)= g_1g_2g_4^2 ,\ \ttheta(\gamma_3)=g_3^2g_4 ,\ \ttheta(\gamma_4)= g_2g_3g_4^2 ,\quad\sigma=g_2$.\\
$V_+=V_6\oplus  V_9\oplus V_{13}$ and $V_-=V_5\oplus V_8\oplus V_{14}\oplus V_{17}$, where $\dim V_i = 1$ for $i =5,6,8,9$, $\dim V_{13} = \dim V_{14} = \dim V_{17} =2.$\\
$(S^2V_+)^\tG=(V_6\otimes V_9)\oplus (S^2V_{13})^\tG$ and $(S^2V_-)^\tG=(S^2V_{14})^\tG$ where all summands have dimension 1.
Condition \eqref{condA} is satisfied, but \eqref{condB1} and \eqref{condB2} are not.
We notice that $V_{14}$ is fixed by $g_3$ and $\Fix\sx{g_3}\xs=V_{13}\oplus V_{14}$. If we define $C':=\tC/\sx g_3\xs$ and $C'':=\tC/\sx g_2,g_3\xs=C'/\sx g_2\xs$ then we have $P(\tC,C)\sim A\times P(C',C'')$ where $A$ is a fixed abelian variety such that $H^{1,0}(A)=V_5\oplus V_8\oplus V_{17}$. 
Therefore the family $P(\tC,C)$ describes a Shimura variety if the family of $P(C',C'')$ moves. And this is true as noted in Example~2 of \cite{cfgp}.

\medskip
\noindent \textbf{Example 5.}\\
$\tg=11$, $g=3$, $b=12$, $\# 2$.\\
$\tG=G(32,42)=(\ZZ/2\times\ZZ/8)\rtimes\ZZ/2=\sx g_1,g_2,g_3,g_4,g_5|g_1^2=g_2^2=g_5^2=1,g_3^2=g_4^2=g_5,g_1^{-1}g_2g_1=g_2g_4,g_1^{-1}g_4g_1=g_4g_5,g_2^{-1}g_4g_2=g_4g_5\xs$,\\
$\ttheta(\gamma_1)=g_2g_4g_5,\ \ttheta(\gamma_2)=g_1,\ \ttheta(\gamma_3)=g_3,\ \ttheta(\gamma_4)=g_1g_2g_3g_4,\quad\sigma=g_5$.\\
$V_+=V_6\oplus V_{10}$ and $V_-= V_{12}\oplus 2V_{13}\oplus V_{14}$ where $\dim V_6=1$ and the others have dimension 2.\\
$(S^2V_+)^\tG=(S^2V_6)\oplus (S^2V_{10})^{\tG}$ and $(S^2V_-)^\tG=(V_{12}\otimes V_{14})^\tG$ where all summands have dimension 1.
Condition \eqref{condA} is satisfied, but \eqref{condB1} and \eqref{condB2} are not.
We consider its centraliser  $K = \sx g_1 g_2, g_3 g_4 \xs \cong \Z/8 \times \Z/2$.  We have $\sx  g_3g_4\xs \lhd  \sx g_3 g_4, g_5\xs \lhd  K \lhd \tG$ and the following diagram.
\[\bCD[column sep=small]
&\tC\ar[ld]\ar[rd]&\\
C'=\tC/\sx g_3g_4\xs\ar[rd]\ar[rdd,"\phi" swap]&&C=\tC/\sx g_5\xs\ar[ld]\\
&F=\tC/\sx g_3 g_4, g_5 \xs \dar&\\
&\tC/K=C'/K'=\PP^1\dar{\eps}\\
&\tC/\tG=\PP^1\\
\eCD\]
where $K'=K/\sx g_3g_4\xs\cong\ZZ/8 = \langle \overline{g_1g_2} \rangle$. Set $h:=  \overline{g_1g_2}$.

Assume that the critical values of the cover $\tC \to \tC/\tG \cong \PP^1$ are $\{ P_1 = 1, P_2 = \infty, P_3 = \lambda, P_4 = 0\}$. 
 The map $\eps$ is a double cover ramified at 2 points, so we can assume $\eps(z) = -z^2 +1$, hence it is ramified at $0, \infty$ and $\eps^{-1} (\lambda) = \{\pm \alpha\}$,  $\eps^{-1} (0) = \{\pm 1\}$ ($\alpha^2 = 1- \lambda$). 

The critical values of $\phi$ are $[\alpha, -\alpha, 1, -1]$ and the monodromy is given by  $[h^2, h^2, h, h^3]$, so the equation of the genus 6 curves $C'$ as  $\Z/8$-cyclic covers of $\mathbb P^1$ is: $y^8 = (x-\alpha)^2 (x+\alpha)^2 (x-1)(x+1)^3$. Hence an equation for the  $\Z/4$ coverings $F\to C'/K'$ is $y^4 = (x-\alpha)^2 (x+\alpha)^2 (x-1)(x+1)^3$, $F$ has genus 2 and this is the Shimura family $(4) = (29)$ of \cite{fgp}. 

Denote by $W_i:= \{ x \in H^0(K_{C'}) \  | \ h(x) = \zeta^i x\}$, where $\zeta$ is a primitive $8$-th root of unity. Then one easily computes that $\dim W_1 = \dim W_4 =0$, $ \dim W_2 = \dim W_3= \dim W_5 =\dim W_6 =1$, $\dim W_7 = 2$. Therefore  we have $H^0(K_F) \cong W_2 \oplus W_6$, $H^{1,0}(P(C',F) \cong W_3 \oplus W_5  \oplus W_7$, hence $(S^2 H^{1,0}(P(C',F) )^{\langle h \rangle}  \cong W_3 \otimes W_5$ is one dimensional. Thus the codifferential of the Prym map of the double coverings $C' \to F$ is the multiplication map $ (S^2 H^{1,0}(P(C',F) )^{\langle h \rangle}  \cong W_3 \otimes W_5 \to H^0(2K_{C'})^{\langle h \rangle}$ and it is an isomorphism. 

This shows that the Pryms $P(C',F)$ vary and yield a Shimura curve in $\A^{\delta}_4$.  In fact this is the first example with $\tg = 6$, $g=2$, $b =6$ in Table 1 and it satisfies both \eqref{condA} and \eqref{condB1}.

Finally the group algebra decomposition gives $ J\tC \sim B_6 \times B_{10}^2 \times B_{12}^2$,  where $\dim B_6 = \dim B_{10} = 1$, $\dim B_{12} = 4$. Moreover $ JC \sim  B_6 \times B_{10}^2$, hence  $P(\tC,C) \sim B_{12}^2$,  $ JC' \sim B_{10}^2 \times B_{12}$, $JF \sim B_{10}^2$, hence  $(P(C',F))^2 \sim B_{12}.$ Therefore $P(\tC,C) \sim  (P(C',F))^2$, hence it moves, and condition \eqref{condB} is satisfied. 

\medskip
\noindent \textbf{Example 6.}\\
$\tg=7$, $g=3$, $b=4$, $\# 4$.\\
$\tG=G(16,6)=\ZZ/8\rtimes\ZZ/2=\sx g_1\xs\rtimes\sx g_2\xs = \sx g_1, g_2 \ | \ g_1^8 = 1, \  g_2^2 = 1, \  g_2^{-1} g_1 g_2 = g_1^5 \xs$. Set $g_3 := g_1^2$, $g_4 = g_1^4$. \\
$\ttheta(\gamma_1)=g_2,\ \ttheta(\gamma_2)=g_2g_4,\ \ttheta(\gamma_3)=g_1g_4,\ \ttheta(\gamma_4)=g_1g_3g_4,\quad\sigma=g_4$.\\
$V_+=V_3\oplus V_5\oplus V_7$ and $V_-= V_9\oplus V_{10}$ where $\dim V_9=\dim V_{10}=2$ and the others have dimension 1.\\
$(S^2V_+)^\tG=(S^2V_3)\oplus (V_5\otimes V_7)$ and $(S^2V_-)^\tG=(V_{9}\otimes V_{10})^\tG$ where all summands have dimension 1.
Condition \eqref{condA} is satisfied, but \eqref{condB1} and \eqref{condB2} are not.
The trace of ${g_2}_{|V_9}$ and of  ${g_2}_{|V_{10}}$ is zero, hence  we can write $V_9=\sx e_1,e_2\xs$, $V_{10}=\sx f_1,f_2\xs$, where $g_2(e_1) = e_1$, $g_2(e_2) = -e_2$, $g_2(f_1) = f_1$, $g_2(f_2) = -f_2$.  Since $g_4$ acts as $-Id$ on $V_9$ and on $V_{10}$,  we have: $g_2g_4(e_1) = -e_1$, $g_2g_4(e_2) = e_2$, $g_2g_4(f_1) = -f_1$, $g_2g_4(f_2) = f_2$. 
The curves $C'=\tC/\sx g_2\xs$ and $D'=\tC/\sx g_2g_4\xs$, are ismorphic since $g_2$ and $g_2g_4$ are conjugated in $\tG$. 
From the group algebra decomposition we get:  $J\tC\sim B_3\times B_5\times B_9^2$,  where $\dim B_3=1$ and $\dim B_5=\dim B_9=2$. Furthemore $JC'\cong JD'\sim B_9$ and $P(\tC,C)\sim JC'\times JD'\sim (JC')^2\sim B_9^2$. In order to study $C'$, we consider $H:=\sx g_3,g_2\xs\cong\ZZ/4\times\ZZ/2$ so that $\sx g_2\xs\lhd H\lhd\tG$ and we get $\tC/H=\PP^1$. We obtain the following diagram
\[\bCD
\tC\ar[rd]\ar[rdd,"k"]\ar[ddd]&\\
&C'=\tC/\sx g_2\xs\dar{t}\\
&\tC/H= \PP^1\ar[ld,"l"]\\
\tC/\tG=\PP^1&
\eCD\]
The map $l$ is a double cover ramified in 2 points, so we may assume that $l(z)=z^2$. Hence the critical values for $k$ are $\{\pm 1,\pm\mu,0,\infty\}$ where $\mu^2=\lambda$. Thus the branch locus of the map $t : C' \to C'/\Z/4$ is $\{1,\mu,0,\infty\}$ with monodromy  $[\bar{2},\bar{2}, \bar{1}, \bar{1}]$, (here by $\bar{n}$ we mean the class  of $n$ in $\Z/4 \cong H/\langle g_2 \rangle$). So the equation for $C'$ is $y^4=(x-1)^2(x-\mu)^2x$ and we conclude that $JC'$ moves.  Hence $P(\tC,C)\sim  (JC')^2$ moves and condition \eqref{condB} is satisfied.

\section*{Appendix}

In this appendix we briefly explain how the \MAGMA\ script works.

The starting data is composed of the number $r$ of critical values on $\PP^1$ and the genus $\tilde g$ of the curve $\tilde C$.
Next we proceed by generating Prym datas with the following steps:
\begin{enumerate}
\item 
We are intrested in the case $r>3$. We consider mainly $r=4$, in this case the maximum order for the authomorphism group of a curve of genus $\tilde g$ is $12(\tilde g-1)$. 
For $r\ge 5$ this limit decreases to $4(\tilde g-1)$.

\item
An ammissible signature $(m_1,\dots,m_r)$ is formed by divisors of $|\tilde G|$ grater than 1 and must satisfy the Riemann-Hurwitz formula
\[2\tilde g-2=|\tilde G|\left(-2+\sum_{i=1}^r\left(1-\frac{1}{m_i}\right)\right).\]
We also restrict to  ordered sequences.

\item
Then we generate the data. For each even order we consider all groups of that cardinality. For each group $\tilde G$ and signature we focus on elements of order $m_1,\dots,m_r$.
We consider all possible $r$-uple $(x_1,\dots,x_r)$ of elements of $\tilde G$ such that $\operatorname{ord}(x_i)=m_i$ and $x_1\cdots x_r=1$.
Moreover all central elements of order 2 are saved.

\item
Finally we  exclude equivalent data. 
Two monodromy data are equivalent  in two cases:  either they are isomorphic under an automorphism of $\tilde G$ or under the action of the Braid group $\braid$.
One datum per orbit is saved.
\end{enumerate}

From a Prym datum $(\tilde G,\tilde \theta,\sigma)$ we obtain the datum $(G,\theta)$ of the quotient curve from the quotient by $\sigma$.
Now we can easly compute the genus $g$ of $C$ and the number $b$ of ramification points via the Riemann-Hurwitz formula or the formula given in Definition~\ref{prymdatum}.
After that we test the hypothesis of  Theorem \ref{teo1ram}. 

\begin{enumerate}[resume]
\item 
We test condition \eqref{condA} by computing the dimension of $S^2H^{1,0}$ for both curves $\tC$, $C$. To do this we use the Chevalley-Weil formula to obtain the $\tilde{G}$-isotopic decomposition of $H^{1,0}$ for both curves.
If condition   \eqref{condA} is verified the datum is recorded.

\item For these data conditions \eqref{condB1} and \eqref{condB2} are tested through the study of the $G$-isotopic decomposition of $S^2V_-$.
If one of them is verified it means that the hypothesis of  Theorem \ref{teo1ram} are satisfied and therefore the datum gives rise to a special variety.
Otherwise further ad-hoc analysis is needed to determine if condition \eqref{condB} is satisfied.

\end{enumerate}

\bigskip

The following tables list all the calculated data. 
Table 1 shows the examples that satisfy both conditions \eqref{condA} and \eqref{condB}. For each pair $\tg$ and $g$ (and thus $b$ and $p = \tg-g$) the examples are listed with a progressive index (under column $\#$). 
For each example it is reported the group $\tG$ and the id of $\tG$ and $G$ as small group in the \MAGMA\ Database. A check mark or a cross tells if the conditions are met or not. The script tests condition \eqref{condB2} only if \eqref{condB1} doesn't hold.
Different examples with same data are grouped in the same row.

In Table 2 for each $\tg$ and $g$  it is reported on the last column the number of non-equivalent data that satisfies condition \eqref{condA}. This includes all examples listed in Table 1 but also all  the cases where  neither \eqref{condB1}, nor  \eqref{condB2} hold  and we did not check  if \eqref{condB} holds. There aren't others cases satisfying \eqref{condA} with $\tg\le 70$.

\def\midline{\addlinespace[0mm]\hline\addlinespace[0.379mm]}
\def\cmidline{\addlinespace[0mm]\cline{2-12}\addlinespace[0.379mm]}
\def\testa{\begin{table}[H]
\begin{tabular}{*{12}{c}}
\toprule
$\tg$&$g$&$b$&$p$&$\#$&$\tG$&$\tG$ \texttt{Id}&$G$ \texttt{Id}&\eqref{condB1}&\eqref{condB2}&$b\ge6$&\eqref{condB}\\ 
\midrule}
\def\coda{\bottomrule
\end{tabular}
\end{table}}
\def\interruzione{\coda \newpage \testa}

\begin{center}
\begin{table}[H]\label{tab1}
\caption{Data satisfying condition \eqref{condA} and \eqref{condB}.}
\begin{tabular}{*{12}{c}}
\toprule
$\tg$&$g$&$b$&$p$&$\#$&$\tG$&$\tG$ \texttt{Id}&$G$ \texttt{Id}&\eqref{condB1}&\eqref{condB2}&$b\ge6$&\eqref{condB}\\ 
\midrule
2 & 0 & 6 & 2 & 1 &$\ZZ/4$& $G(4, 1)$ & $G(2, 1)$ & \checkmark & & \checkmark & \checkmark \\ 
2 & 0 & 6 & 2 & 2 &$\ZZ/6$& $G(6, 2)$ & $G(3, 1)$ & \checkmark & & \checkmark & \checkmark \\ 
2 & 0 & 6 & 2 & 3 &$D_4$& $G(8, 3)$ & $G(4, 2)$ & $X$ & $X$ & \checkmark & \checkmark \\ 
2 & 0 & 6 & 2 & 4 &$D_6$& $G(12, 4)$ & $G(6, 1)$ & $X$ & $X$ & \checkmark & \checkmark \\  
\midline
3 & 2 & 0 & 1 & 1 &$\ZZ/2\times\ZZ/4$& $G(8, 2)$ & $G(4, 1)$ & \checkmark & & $X$ & \checkmark \\ 
3 & 2 & 0 & 1 & 2 &$\ZZ/2\times D_4$& $G(16, 11)$ & $G(8, 3)$ & \checkmark & & $X$ & \checkmark \\ 
\cmidline
3 & 1 & 4 & 2 & 1 &$\ZZ/4$& $G(4, 1)$ & $G(2, 1)$ & \checkmark & & $X$ & \checkmark \\ 
3 & 1 & 4 & 2 & 2 &$\ZZ/6$& $G(6, 2)$ & $G(3, 1)$ & \checkmark & & $X$ & \checkmark \\ 
3 & 1 & 4 & 2 & 3,4,6 &$\ZZ/2\times\ZZ/4$& $G(8, 2)$ & $G(4, 1)$ & \checkmark & & $X$ & \checkmark \\ 
3 & 1 & 4 & 2 & 5 &$\ZZ/2\times\ZZ/4$& $G(8, 2)$ & $G(4, 2)$ & \checkmark & & $X$ & \checkmark \\ 
3 & 1 & 4 & 2 & 8 &$\ZZ/2\times D_4$& $G(16, 11)$ & $G(8, 5)$ & $X$ & \checkmark & $X$ & \checkmark \\ 
\cmidline
3 & 0 & 8 & 3 & 1 &$\ZZ/2\times\ZZ/4$& $G(8, 2)$ & $G(4, 2)$ & \checkmark & & \checkmark & \checkmark \\ 
\midline
4 & 2 & 2 & 2 & 1 &$\ZZ/6$& $G(6, 2)$ & $G(3, 1)$ & \checkmark & & $X$ & \checkmark \\ 
\cmidline
4 & 1 & 6 & 3 & 1 &$\ZZ/6$& $G(6, 2)$ & $G(3, 1)$ & \checkmark & & \checkmark & \checkmark \\ 
4 & 1 & 6 & 3 & 2,3 &$\ZZ/2\times\ZZ/6$& $G(12, 5)$ & $G(6, 2)$ & \checkmark & & \checkmark & \checkmark \\ 
\cmidline
4 & 0 & 10 & 4 & 1 &$Q_8$& $G(8, 4)$ & $G(4, 2)$ & $X$ & $X$ & \checkmark & \checkmark \\ 
\midline
5 & 3 & 0 & 2 & 1--3 &$\ZZ/2\times\ZZ/4$& $G(8, 2)$ & $G(4, 1)$ & \checkmark & & $X$ & \checkmark \\ 
5 & 3 & 0 & 2 & 4 &$\ZZ/2\times\ZZ/6$& $G(12, 5)$ & $G(6, 2)$ & \checkmark & & $X$ & \checkmark \\ 
5 & 3 & 0 & 2 & 5 &$(\ZZ/2\times\ZZ/2)\rtimes\ZZ/4$& $G(16, 3)$ & $G(8, 3)$ & \checkmark & & $X$ & \checkmark \\ 
5 & 3 & 0 & 2 & 8 &$(\ZZ/2\times\ZZ/2)\rtimes\ZZ/4$& $G(16, 3)$ & $G(8, 2)$ & $X$ & \checkmark & $X$ & \checkmark \\ 
5 & 3 & 0 & 2 & 9--11 &$\ZZ/2\times\ZZ/2\times\ZZ/4$& $G(16, 10)$ & $G(8, 2)$ & \checkmark & & $X$ & \checkmark \\ 
5 & 3 & 0 & 2 & 14 &$\ZZ/2\times\ZZ/2\times S_3$& $G(24, 14)$ & $G(12, 4)$ & $X$ & \checkmark & $X$ & \checkmark \\ 
5 & 3 & 0 & 2 & 15 &$\ZZ/2\times A_4$& $G(24, 13)$ & $G(12, 3)$ & \checkmark & & $X$ & \checkmark \\ 
5 & 3 & 0 & 2 & 16,17 &$(\ZZ/2)^2\wr\ZZ/2$& $G(32, 27)$ & $G(16, 11)$ & $X$ & \checkmark & $X$ & \checkmark \\ 
5 & 3 & 0 & 2 & 18 &$\ZZ/4\times D_4$& $G(32, 28)$ & $G(16, 13)$ & $X$ & \checkmark & $X$ & \checkmark \\ 
5 & 3 & 0 & 2 & 19 &$\ZZ/2\times S_4$& $G(48, 48)$ & $G(24, 12)$ & $X$ & \checkmark & $X$ & \checkmark \\ 
\cmidline
5 & 2 & 4 & 3 & 1,2 &$(\ZZ/2\times\ZZ/2)\rtimes\ZZ/4$& $G(16, 3)$ & $G(8, 3)$ & $X$ & \checkmark & $X$ & \checkmark \\ 
5 & 2 & 4 & 3 & 3 &$(\ZZ/2\times\ZZ/2)\rtimes\ZZ/4$& $G(16, 3)$ & $G(8, 3)$ & \checkmark & & $X$ & \checkmark \\ 
\cmidline
5 & 1 & 8 & 4 & 1 &$\ZZ/8$& $G(8, 1)$ & $G(4, 1)$ & \checkmark & & \checkmark & \checkmark \\ 
5 & 1 & 8 & 4 & 2 &$\ZZ/2\times\ZZ/4$& $G(8, 2)$ & $G(4, 2)$ & \checkmark & & \checkmark & \checkmark \\ 
5 & 1 & 8 & 4 & 3 &$Q_8$& $G(8, 4)$ & $G(4, 2)$ & $X$ & $X$ & \checkmark & \checkmark \\ 
5 & 1 & 8 & 4 & 4 &$(\ZZ/2\times\ZZ/2)\rtimes\ZZ/4$& $G(16, 3)$ & $G(8, 3)$ & $X$ & $X$ & \checkmark & \checkmark \\ 
5 & 1 & 8 & 4 & 5 &$Q_8\rtimes\ZZ/2$& $G(16, 8)$ & $G(8, 3)$ & $X$ & $X$ & \checkmark & \checkmark \\ 
5 & 1 & 8 & 4 & 6 &$\ZZ/2\times\ZZ/2\times\ZZ/4$& $G(16, 10)$ & $G(8, 5)$ & \checkmark & & \checkmark & \checkmark \\ 
5 & 1 & 8 & 4 & 7 &$\ZZ/2\circ D_4$& $G(16, 13)$ & $G(8, 5)$ & $X$ & $X$ & \checkmark & \checkmark \\ 
5 & 1 & 8 & 4 & 8 &$\ZZ/4\times D_4$& $G(32, 28)$ & $G(16, 11)$ & $X$ & \checkmark & \checkmark & \checkmark \\ 
5 & 1 & 8 & 4 & 9 &$\ZZ/8\rtimes(\ZZ/2\times\ZZ/2)$& $G(32, 43)$ & $G(16, 11)$ & $X$ & $X$ & \checkmark & \checkmark \\ 
\cmidline
5 & 0 & 12 & 5 & 1 &$\ZZ/3\rtimes\ZZ/4$& $G(12, 1)$ & $G(6, 1)$ & $X$ & $X$ & \checkmark & \checkmark \\ 
\midline
6 & 3 & 2 & 3 & 1,2 &$\ZZ/2\times\ZZ/6$& $G(12, 5)$ & $G(6, 2)$ & \checkmark & & $X$ & \checkmark \\ 
\cmidline
6 & 2 & 6 & 4 & 1 &$\ZZ/8$& $G(8, 1)$ & $G(4, 1)$ & \checkmark & & \checkmark & \checkmark \\ 
6 & 2 & 6 & 4 & 2 &$\ZZ/10$& $G(10, 2)$ & $G(5, 1)$ & \checkmark & & \checkmark & \checkmark \\ 
6 & 2 & 6 & 4 & 3 &$\ZZ/3\rtimes\ZZ/4$& $G(12, 1)$ & $G(6, 1)$ & $X$ & $X$ & \checkmark & \checkmark \\ 
6 & 2 & 6 & 4 & 4 &$\ZZ/3\rtimes D_4$& $G(24, 8)$ & $G(12, 4)$ & $X$ & $X$ & \checkmark & \checkmark \\ 
\interruzione
7 & 4 & 0 & 3 & 1 &$\ZZ/2\times\ZZ/6$& $G(12, 5)$ & $G(6, 2)$ & \checkmark & & $X$ & \checkmark \\ 
7 & 4 & 0 & 3 & 2 &$\ZZ/4\rtimes\ZZ/4$& $G(16, 4)$ & $G(8, 4)$ & \checkmark & & $X$ & \checkmark \\ 
7 & 4 & 0 & 3 & 3 &$\ZZ/4\rtimes\ZZ/4$& $G(16, 4)$ & $G(8, 4)$ & $X$ & \checkmark & $X$ & \checkmark \\ 
7 & 4 & 0 & 3 & 4 &$\ZZ/2\times Q_8$& $G(16, 12)$ & $G(8, 4)$ & \checkmark & & $X$ & \checkmark \\ 
\cmidline
7 & 3 & 4 & 4 & 1 &$\ZZ/8$& $G(8, 1)$ & $G(4, 1)$ & \checkmark & & $X$ & \checkmark \\ 
7 & 3 & 4 & 4 & 2 &$\ZZ/2\times\ZZ/6$& $G(12, 5)$ & $G(6, 2)$ & \checkmark & & $X$ & \checkmark \\ 
7 & 3 & 4 & 4 & 3 &$\ZZ/12$& $G(12, 2)$ & $G(6, 2)$ & \checkmark & & $X$ & \checkmark \\ 
7 & 3 & 4 & 4 & 4 &$M_4(2)$& $G(16, 6)$ & $G(8, 2)$ & $X$ & $X$ & $X$ & \checkmark \\ 
7 & 3 & 4 & 4 & 6,7 &$\ZZ/4\times\ZZ/4$& $G(16, 2)$ & $G(8, 2)$ & \checkmark & & $X$ & \checkmark \\ 
7 & 3 & 4 & 4 & 8,9 &$\ZZ/4\rtimes\ZZ/4$& $G(16, 4)$ & $G(8, 2)$ & $X$ & \checkmark & $X$ & \checkmark \\ 
\cmidline
7 & 2 & 8 & 5 & 1,3 &$\ZZ/12$& $G(12, 2)$ & $G(6, 2)$ & \checkmark & & \checkmark & \checkmark \\ 
7 & 2 & 8 & 5 & 2 &$\ZZ/3\rtimes\ZZ/4$& $G(12, 1)$ & $G(6, 1)$ & $X$ & $X$ & \checkmark & \checkmark \\ 
7 & 2 & 8 & 5 & 4 &$\ZZ/4\rtimes\ZZ/4$& $G(16, 4)$ & $G(8, 3)$ & $X$ & \checkmark & \checkmark & \checkmark \\ 
7 & 2 & 8 & 5 & 5 &$\ZZ/4\rtimes\ZZ/4$& $G(16, 4)$ & $G(8, 3)$ & \checkmark & & \checkmark & \checkmark \\ 
\cmidline
7 & 1 & 12 & 6 & 1 &$\ZZ/3\rtimes\ZZ/4$& $G(12, 1)$ & $G(6, 1)$ & $X$ & $X$ & \checkmark & \checkmark \\ 
7 & 1 & 12 & 6 & 2 &$\ZZ/2\times Q_8$& $G(16, 12)$ & $G(8, 5)$ & $X$ & \checkmark & \checkmark & \checkmark \\ 
7 & 1 & 12 & 6 & 3 &$\ZZ/4\times S_3$& $G(24, 5)$ & $G(12, 4)$ & $X$ & $X$ & \checkmark & \checkmark \\ 
7 & 1 & 12 & 6 & 4 &$Q_8\rtimes\ZZ/3$& $G(24, 3)$ & $G(12, 3)$ & $X$ & $X$ & \checkmark & \checkmark \\ 
7 & 1 & 12 & 6 & 5 &$Q_8\rtimes_3S_3$& $G(48, 41)$ & $G(24, 14)$ & $X$ & $X$ & \checkmark & \checkmark \\ 
\midline
8 & 4 & 2 & 4 & 1,2 &$\ZZ/10$& $G(10, 2)$ & $G(5, 1)$ & \checkmark & & $X$ & \checkmark \\ 
8 & 4 & 2 & 4 & 3,4 &$\ZZ/2\times\ZZ/6$& $G(12, 5)$ & $G(6, 2)$ & \checkmark & & $X$ & \checkmark \\ 
\midline
9 & 5 & 0 & 4 & 1--5 &$\ZZ/2\times\ZZ/8$& $G(16, 5)$ & $G(8, 1)$ & \checkmark & & $X$ & \checkmark \\ 
9 & 5 & 0 & 4 & 6 &$\ZZ/4\times\ZZ/4$& $G(16, 2)$ & $G(8, 2)$ & \checkmark & & $X$ & \checkmark \\ 
9 & 5 & 0 & 4 & 8 &$\ZZ/4\rtimes\ZZ/4$& $G(16, 4)$ & $G(8, 4)$ & \checkmark & & $X$ & \checkmark \\ 
9 & 5 & 0 & 4 & 11--13 &$\ZZ/2\times\ZZ/2\times\ZZ/4$& $G(16, 10)$ & $G(8, 2)$ & \checkmark & & $X$ & \checkmark \\ 
9 & 5 & 0 & 4 & 19,20 &$\ZZ/2\times\ZZ/2\times\ZZ/6$& $G(24, 15)$ & $G(12, 5)$ & \checkmark & & $X$ & \checkmark \\ 
9 & 5 & 0 & 4 & 21 &$\ZZ/6\rtimes\ZZ/4$& $G(24, 7)$ & $G(12, 1)$ & $X$ & \checkmark & $X$ & \checkmark \\ 
9 & 5 & 0 & 4 & 30 &$\ZZ/2\times(\ZZ/2)^2\rtimes\ZZ/4$& $G(32, 22)$ & $G(16, 11)$ & \checkmark & & $X$ & \checkmark \\ 
9 & 5 & 0 & 4 & 34 &$\ZZ/4\times D_4$& $G(32, 25)$ & $G(16, 13)$ & \checkmark & & $X$ & \checkmark \\ 
9 & 5 & 0 & 4 & 35 &$\ZZ/4\times D_4$& $G(32, 25)$ & $G(16, 10)$ & $X$ & \checkmark & $X$ & \checkmark \\ 
9 & 5 & 0 & 4 & 37 &$(\ZZ/2)^2\mathbin.D_4$& $G(32, 30)$ & $G(16, 13)$ & $X$ & \checkmark & $X$ & \checkmark \\ 
9 & 5 & 0 & 4 & 44,45 &$\ZZ/2\times\ZZ/3\rtimes D_4$& $G(48, 43)$ & $G(24, 8)$ & $X$ & \checkmark & $X$ & \checkmark \\ 
9 & 5 & 0 & 4 & 48--50 &$(\ZZ/2)^3\rtimes_2D_4$& $G(64, 73)$ & $G(32, 27)$ & $X$ & \checkmark & $X$ & \checkmark \\ 
9 & 5 & 0 & 4 & 55 &$\ZZ/4\rtimes D_8$& $G(64, 140)$ & $G(32, 43)$ & $X$ & \checkmark & $X$ & \checkmark \\ 
\cmidline
9 & 4 & 4 & 5 & 1 &$\ZZ/10$& $G(10, 2)$ & $G(5, 1)$ & \checkmark & & $X$ & \checkmark \\ 
9 & 4 & 4 & 5 & 2,3 &$\ZZ/12$& $G(12, 2)$ & $G(6, 2)$ & \checkmark & & $X$ & \checkmark \\ 
\cmidline
9 & 3 & 8 & 6 & 1 &$\ZZ/12$& $G(12, 2)$ & $G(6, 2)$ & \checkmark & & \checkmark & \checkmark \\ 
9 & 3 & 8 & 6 & 2,3 &$\ZZ/2\times\ZZ/8$& $G(16, 5)$ & $G(8, 2)$ & \checkmark & & \checkmark & \checkmark \\ 
9 & 3 & 8 & 6 & 4 &$\ZZ/4\times\ZZ/4$& $G(16, 2)$ & $G(8, 2)$ & \checkmark & & \checkmark & \checkmark \\ 
9 & 3 & 8 & 6 & 5 &$\ZZ/4\rtimes\ZZ/4$& $G(16, 4)$ & $G(8, 3)$ & $X$ & $X$ & \checkmark & \checkmark \\ 
9 & 3 & 8 & 6 & 6 &$\ZZ/4\times S_3$& $G(24, 5)$ & $G(12, 4)$ & $X$ & $X$ & \checkmark & \checkmark \\ 
9 & 3 & 8 & 6 & 7 &$\ZZ/4\times S_3$& $G(24, 5)$ & $G(12, 4)$ & \checkmark & & \checkmark & \checkmark \\ 
9 & 3 & 8 & 6 & 8 &$\ZZ/4\times D_4$& $G(32, 25)$ & $G(16, 11)$ & \checkmark & & \checkmark & \checkmark \\ 
9 & 3 & 8 & 6 & 9 &$(\ZZ/2)^2\mathbin.D_4$& $G(32, 30)$ & $G(16, 11)$ & $X$ & $X$ & \checkmark & \checkmark \\ 
\cmidline
9 & 2 & 12 & 7 & 1 &$\ZZ/6\rtimes\ZZ/4$& $G(24, 7)$ & $G(12, 4)$ & $X$ & \checkmark & \checkmark & \checkmark \\ 
\interruzione
10 & 4 & 6 & 6 & 1,2 &$\ZZ/12$& $G(12, 2)$ & $G(6, 2)$ & \checkmark & & \checkmark & \checkmark \\ 
10 & 4 & 6 & 6 & 3,4 &$\ZZ/3\times\ZZ/6$& $G(18, 5)$ & $G(9, 2)$ & \checkmark & & \checkmark & \checkmark \\ 
10 & 4 & 6 & 6 & 5 &$\ZZ/3\times D_4$& $G(24, 10)$ & $G(12, 5)$ & $X$ & \checkmark & \checkmark & \checkmark \\ 
10 & 4 & 6 & 6 & 6 &$Q_8\rtimes\ZZ/3$& $G(24, 3)$ & $G(12, 3)$ & $X$ & $X$ & \checkmark & \checkmark \\ 
10 & 4 & 6 & 6 & 7,9 &$\ZZ/6\times S_3$& $G(36, 12)$ & $G(18, 3)$ & $X$ & $X$ & \checkmark & \checkmark \\ 
10 & 4 & 6 & 6 & 8 &$\ZZ/6\times S_3$& $G(36, 12)$ & $G(18, 3)$ & $X$ & \checkmark & \checkmark & \checkmark \\ 
10 & 4 & 6 & 6 & 10 &$Q_8\rtimes S_3$& $G(48, 29)$ & $G(24, 12)$ & $X$ & $X$ & \checkmark & \checkmark \\ 
\midline
11 & 6 & 0 & 5 & 1 &$\ZZ/2\times\ZZ/8$& $G(16, 5)$ & $G(8, 1)$ & \checkmark & & $X$ & \checkmark \\ 
11 & 6 & 0 & 5 & 2 &$\ZZ/2\times\ZZ/12$& $G(24, 9)$ & $G(12, 2)$ & \checkmark & & $X$ & \checkmark \\ 
11 & 6 & 0 & 5 & 3,4 &$\ZZ/6\rtimes\ZZ/4$& $G(24, 7)$ & $G(12, 1)$ & $X$ & \checkmark & $X$ & \checkmark \\ 
\cmidline
11 & 3 & 12 & 8 & 1 &$\ZZ/6\rtimes\ZZ/4$& $G(24, 7)$ & $G(12, 4)$ & $X$ & \checkmark & \checkmark & \checkmark \\ 
11 & 3 & 12 & 8 & 2 &$\ZZ/4\circ D_8$& $G(32, 42)$ & $G(16, 11)$ & $X$ & $X$ & \checkmark & \checkmark \\ 
11 & 3 & 12 & 8 & 3 &$\ZZ/8\mathbin.(\ZZ/2)^2$& $G(32, 44)$ & $G(16, 11)$ & $X$ & $X$ & \checkmark & \checkmark \\ 
\midline
12 & 6 & 2 & 6 & 1,2 &$\ZZ/14$& $G(14, 2)$ & $G(7, 1)$ & \checkmark & & $X$ & \checkmark \\ 
12 & 6 & 2 & 6 & 3 &$\ZZ/2\times\ZZ/10$& $G(20, 5)$ & $G(10, 2)$ & \checkmark & & $X$ & \checkmark \\ 
\midline
13 & 7 & 0 & 6 & 1 &$\ZZ/2\times\ZZ/8$& $G(16, 5)$ & $G(8, 1)$ & \checkmark & & $X$ & \checkmark \\ 
13 & 7 & 0 & 6 & 2 &$\ZZ/2\times\ZZ/10$& $G(20, 5)$ & $G(10, 2)$ & \checkmark & & $X$ & \checkmark \\ 
13 & 7 & 0 & 6 & 3 &$\ZZ/2\times\ZZ/12$& $G(24, 9)$ & $G(12, 2)$ & \checkmark & & $X$ & \checkmark \\ 
13 & 7 & 0 & 6 & 6 &$D_4\rtimes\ZZ/4$& $G(32, 9)$ & $G(16, 7)$ & \checkmark & & $X$ & \checkmark \\ 
13 & 7 & 0 & 6 & 11 &$(\ZZ/4\times\ZZ/4)\rtimes\ZZ/2$& $G(32, 24)$ & $G(16, 13)$ & \checkmark & & $X$ & \checkmark \\ 
13 & 7 & 0 & 6 & 12 &$(\ZZ/2\times\ZZ/2)\rtimes Q_8$& $G(32, 29)$ & $G(16, 13)$ & $X$ & \checkmark & $X$ & \checkmark \\ 
13 & 7 & 0 & 6 & 17 &$D_4\rtimes D_4$& $G(64, 130)$ & $G(32, 43)$ & $X$ & \checkmark & $X$ & \checkmark \\ 
\cmidline
13 & 6 & 4 & 7 & 1 &$\ZZ/18$& $G(18, 2)$ & $G(9, 1)$ & \checkmark & & $X$ & \checkmark \\ 
13 & 6 & 4 & 7 & 2 &$\ZZ/2\times\ZZ/12$& $G(24, 9)$ & $G(12, 5)$ & \checkmark & & $X$ & \checkmark \\ 
13 & 6 & 4 & 7 & 3 &$\ZZ/2\times\ZZ/12$& $G(24, 9)$ & $G(12, 2)$ & \checkmark & & $X$ & \checkmark \\ 
\cmidline
13 & 5 & 8 & 8 & 1 &$Q_8\rtimes\ZZ/3$& $G(24, 3)$ & $G(12, 3)$ & $X$ & $X$ & \checkmark & \checkmark \\ 
13 & 5 & 8 & 8 & 2 &$(\ZZ/2\times\ZZ/2)\rtimes\ZZ/8$& $G(32, 5)$ & $G(16, 3)$ & $X$ & $X$ & \checkmark & \checkmark \\ 
13 & 5 & 8 & 8 & 3 &$\ZZ/8\circ D_4$& $G(32, 38)$ & $G(16, 10)$ & $X$ & $X$ & \checkmark & \checkmark \\ 
13 & 5 & 8 & 8 & 4 &$\ZZ/2\times {\rm SD}_{16}$& $G(32, 40)$ & $G(16, 11)$ & $X$ & $X$ & \checkmark & \checkmark \\ 
13 & 5 & 8 & 8 & 5 &$(\ZZ/2\times\ZZ/2)\rtimes Q_8$& $G(32, 29)$ & $G(16, 11)$ & $X$ & \checkmark & \checkmark & \checkmark \\ 
13 & 5 & 8 & 8 & 6 &$\ZZ/4\mathbin.A_4$& $G(48, 33)$ & $G(24, 13)$ & $X$ & $X$ & \checkmark & \checkmark \\ 
13 & 5 & 8 & 8 & 7 &$D_4\rtimes D_4$& $G(64, 130)$ & $G(32, 27)$ & $X$ & $X$ & \checkmark & \checkmark \\ 
\midline
14 & 7 & 2 & 7 & 1 &$\ZZ/18$& $G(18, 2)$ & $G(9, 1)$ & \checkmark & & $X$ & \checkmark \\ 
\cmidline
14 & 6 & 6 & 8 & 1 &$\ZZ/16$& $G(16, 1)$ & $G(8, 1)$ & \checkmark & & \checkmark & \checkmark \\ 
\midline
15 & 8 & 0 & 7 & 1,2 &$\ZZ/2\times\ZZ/12$& $G(24, 9)$ & $G(12, 2)$ & \checkmark & & $X$ & \checkmark \\ 
\cmidline
15 & 7 & 4 & 8 & 1 &$\ZZ/16$& $G(16, 1)$ & $G(8, 1)$ & \checkmark & & $X$ & \checkmark \\ 
15 & 7 & 4 & 8 & 2 &$\ZZ/2\times\ZZ/12$& $G(24, 9)$ & $G(12, 5)$ & \checkmark & & $X$ & \checkmark \\ 
\cmidline
15 & 5 & 12 & 10 & 1 &$\ZZ/4\times D_5$& $G(40, 5)$ & $G(20, 4)$ & $X$ & $X$ & \checkmark & \checkmark \\ 
15 & 5 & 12 & 10 & 2 &$\ZZ/4\circ D_{12}$& $G(48, 37)$ & $G(24, 14)$ & $X$ & $X$ & \checkmark & \checkmark \\ 
15 & 5 & 12 & 10 & 3 &$D_4\rtimes_2S_3$& $G(48, 39)$ & $G(24, 14)$ & $X$ & $X$ & \checkmark & \checkmark \\ 
\midline

17 & 9 & 0 & 8 & 1 &$\ZZ/2\times\ZZ/12$& $G(24, 9)$ & $G(12, 2)$ & \checkmark & & $X$ & \checkmark \\ 
17 & 9 & 0 & 8 & 3 &$\ZZ/2\times\ZZ/4\times\ZZ/4$& $G(32, 21)$ & $G(16, 2)$ & \checkmark & & $X$ & \checkmark \\ 
17 & 9 & 0 & 8 & 4 &$\ZZ/4\times Q_8$& $G(32, 26)$ & $G(16, 10)$ & $X$ & \checkmark & $X$ & \checkmark \\ 
17 & 9 & 0 & 8 & 6 &$D_6\rtimes\ZZ/4$& $G(48, 14)$ & $G(24, 6)$ & \checkmark & & $X$ & \checkmark \\ 
17 & 9 & 0 & 8 & 12 &$A_4\rtimes\ZZ/4$& $G(48, 30)$ & $G(24, 12)$ & $X$ & \checkmark & $X$ & \checkmark \\ 
17 & 9 & 0 & 8 & 17 &$(\ZZ/2)^4\mathbin._3(\ZZ/2)^2$& $G(64, 71)$ & $G(32, 34)$ & \checkmark & & $X$ & \checkmark \\ 
17 & 9 & 0 & 8 & 31 &$M_4(2)\rtimes D_4$& $G(128, 738)$ & $G(64, 134)$ & $X$ & \checkmark & $X$ & \checkmark \\ 
\cmidline
17 & 5 & 16 & 12 & 1 &$D_4\mathbin._{10}D_4$& $G(64, 137)$ & $G(32, 27)$ & $X$ & $X$ & \checkmark & \checkmark \\ 
\midline
19 & 9 & 4 & 10 & 1 &$\ZZ/20$& $G(20, 2)$ & $G(10, 2)$ & \checkmark & & $X$ & \checkmark \\ 
19 & 9 & 4 & 10 & 2 &$\ZZ/4\times D_5$& $G(40, 5)$ & $G(20, 4)$ & \checkmark & & $X$ & \checkmark \\ 
\interruzione
21 & 11 & 0 & 10 & 1 &$\ZZ/4\times\ZZ/8$& $G(32, 3)$ & $G(16, 5)$ & \checkmark & & $X$ & \checkmark \\ 
21 & 11 & 0 & 10 & 5 &$\ZZ/6\mathbin.D_4$& $G(48, 19)$ & $G(24, 8)$ & $X$ & \checkmark & $X$ & \checkmark \\ 
21 & 11 & 0 & 10 & 6 &$\ZZ/4\times D_8$& $G(64, 118)$ & $G(32, 39)$ & \checkmark & & $X$ & \checkmark \\ 
\cmidline
21 & 9 & 8 & 12 & 1 &$S_3\times\ZZ/8$& $G(48, 4)$ & $G(24, 5)$ & $X$ & $X$ & \checkmark & \checkmark \\ 
21 & 9 & 8 & 12 & 2 &$\ZZ/24\rtimes\ZZ/2$& $G(48, 6)$ & $G(24, 6)$ & $X$ & $X$ & \checkmark & \checkmark \\ 
21 & 9 & 8 & 12 & 3 &$\ZZ/8\circ D_8$& $G(64, 124)$ & $G(32, 25)$ & $X$ & $X$ & \checkmark & \checkmark \\ 
21 & 9 & 8 & 12 & 4 &$D_4\mathbin._7D_4$& $G(64, 133)$ & $G(32, 27)$ & $X$ & $X$ & \checkmark & \checkmark \\ 
21 & 9 & 8 & 12 & 5 &$\ZZ/8\mathbin._{12}D_4$& $G(64, 176)$ & $G(32, 34)$ & $X$ & $X$ & \checkmark & \checkmark \\ 
21 & 9 & 8 & 12 & 6 &$\ZZ/4\mathbin._6S_4$& $G(96, 192)$ & $G(48, 48)$ & $X$ & $X$ & \checkmark & \checkmark \\ 
\midline

25 & 13 & 0 & 12 & 6 &$(\ZZ/2)^3\rtimes Q_8$& $G(64, 74)$ & $G(32, 27)$ & $X$ & \checkmark & $X$ & \checkmark \\ 
\cmidline
25 & 9 & 16 & 16 & 1 &$(\ZZ/2\times \ZZ/8)\mathbin._2D_4$& $G(128, 749)$ & $G(64, 73)$ & $X$ & $X$ & \checkmark & \checkmark \\ 
\midline
29 & 13 & 8 & 16 & 1 &$Q_{16}\mathbin.D_4$& $G(128, 925)$ & $G(64, 128)$ & $X$ & $X$ & \checkmark & \checkmark \\ 
\midline
31 & 13 & 12 & 18 & 1 &$D_6\mathbin.D_6$& $G(144, 141)$ & $G(72, 46)$ & $X$ & $X$ & \checkmark & \checkmark \\ 





\coda
\end{center}

\begin{center}
\begin{table}[H]\label{tab2}
\caption{Number of data satisfying condition \eqref{condA}.}
\begin{tabular}{*{5}{c}}
\toprule
$\tg$&$g$&$b$&$p$&$\#$\\
\midrule
2 & 0 & 6 & 2 & 4 \\ \hline
3 & 2 & 0 & 1 & 2 \\ 
3 & 1 & 4 & 2 & 8 \\ 
3 & 0 & 8 & 3 & 1 \\ \hline
4 & 2 & 2 & 2 & 2 \\ 
4 & 1 & 6 & 3 & 3 \\ 
4 & 0 & 10 & 4 & 1 \\ \hline
5 & 3 & 0 & 2 & 19 \\ 
5 & 2 & 4 & 3 & 3 \\ 
5 & 1 & 8 & 4 & 9 \\ 
5 & 0 & 12 & 5 & 1 \\ \hline
6 & 3 & 2 & 3 & 2 \\ 
6 & 2 & 6 & 4 & 4 \\ \hline
7 & 4 & 0 & 3 & 5 \\ 
7 & 3 & 4 & 4 & 11 \\ 
7 & 2 & 8 & 5 & 5 \\ 
7 & 1 & 12 & 6 & 5 \\ \hline
8 & 4 & 2 & 4 & 5 \\ 
\bottomrule
\end{tabular}\qquad
\begin{tabular}{*{5}{c}}
\toprule
$\tg$&$g$&$b$&$p$&$\#$\\ 
\midrule
9 & 5 & 0 & 4 & 58 \\ 
9 & 4 & 4 & 5 & 3 \\ 
9 & 3 & 8 & 6 & 9 \\ 
9 & 2 & 12 & 7 & 1 \\ \hline
10 & 4 & 6 & 6 & 10 \\ \hline
11 & 6 & 0 & 5 & 6 \\ 
11 & 5 & 4 & 6 & 2 \\ 
11 & 3 & 12 & 8 & 3 \\ \hline
12 & 6 & 2 & 6 & 3 \\ \hline
13 & 7 & 0 & 6 & 22 \\ 
13 & 6 & 4 & 7 & 3 \\ 
13 & 5 & 8 & 8 & 7 \\ \hline
14 & 7 & 2 & 7 & 1 \\ 
14 & 6 & 6 & 8 & 1 \\ \hline
15 & 8 & 0 & 7 & 2 \\ 
15 & 7 & 4 & 8 & 3 \\ 
15 & 5 & 12 & 10 & 3 \\ \hline
16 & 8 & 2 & 8 & 2 \\ 
\bottomrule
\end{tabular}\qquad
\begin{tabular}{*{5}{c}}
\toprule
$\tg$&$g$&$b$&$p$&$\#$\\ 
\midrule
17 & 9 & 0 & 8 & 35 \\ 
17 & 5 & 16 & 12 & 1 \\ \hline
19 & 10 & 0 & 9 & 1 \\ 
19 & 9 & 4 & 10 & 4 \\ \hline
21 & 11 & 0 & 10 & 10 \\ 
21 & 9 & 8 & 12 & 6 \\ \hline
23 & 11 & 4 & 12 & 4 \\ \hline
25 & 13 & 0 & 12 & 18 \\ 
25 & 9 & 16 & 16 & 1 \\ \hline
29 & 15 & 0 & 14 & 4 \\ 
29 & 13 & 8 & 16 & 1 \\ \hline
31 & 15 & 4 & 16 & 1 \\ 
31 & 13 & 12 & 18 & 1 \\ \hline
33 & 17 & 0 & 16 & 5 \\ \hline
39 & 19 & 4 & 20 & 1 \\ \hline
41 & 21 & 0 & 20 & 2 \\ \hline
47 & 23 & 4 & 24 & 1 \\ \hline
49 & 25 & 0 & 24 & 2 \\ 
\bottomrule
\end{tabular}
\end{table}
\end{center}

\end{document}